\documentclass[reqno, 12pt, reqno]{amsart}

\usepackage{amsfonts, amsthm, amsmath, amssymb}
\usepackage{hyperref}
\hypersetup{colorlinks=false}

\usepackage[margin=1.5in]{geometry}

\RequirePackage{mathrsfs} \let\mathcal\mathscr
\numberwithin{equation}{section}

\renewcommand{\phi}{\varphi}
\renewcommand{\rho}{\varrho}

\newcommand{\0}{\mathbf{0}}
\newcommand{\PP}{\mathbb{P}}

\newcommand{\FF}{\mathbb{F}}
\newcommand{\ZZ}{\mathbb{Z}}

\newcommand{\NN}{\mathbb{N}}
\newcommand{\QQ}{\mathbb{Q}}
\newcommand{\RR}{\mathbb{R}}

\newcommand{\val}{{\rm val}}

\renewcommand{\leq}{\leqslant}

\renewcommand{\geq}{\geqslant}

\newcommand{\h}{\mathbf{h}}

\newcommand{\m}{\mathbf{m}}

\newcommand{\x}{\mathbf{x}}

\renewcommand{\c}{\mathbf{c}}
\renewcommand{\v}{\mathbf{v}}
\renewcommand{\u}{\mathbf{u}}

\newcommand{\w}{\mathbf{w}}
\renewcommand{\b}{\mathbf{b}}

\renewcommand{\k}{\mathbf{k}}
\newcommand{\s}{\mathbf{s}}

\newcommand{\ve}{\varepsilon}

\renewcommand{\t}{\mathbf{t}}

\newcommand{\1}{\mathbf{1}}

\newcommand{\rd}{\mathrm{d}}

\newtheorem{thm}{Theorem}[section]

\newtheorem{lem}[thm]{Lemma}

\theoremstyle{definition}

\numberwithin{equation}{section}

\begin{document}
\date{\today}

\title[Arithmetic of  higher-dimensional orbifolds]{Arithmetic of higher-dimensional orbifolds and a mixed Waring problem}

\author{Tim Browning}

\address{IST Austria\\
Am Campus 1\\
3400 Klosterneuburg\\
Austria}
\email{tdb@ist.ac.at}

\author{Shuntaro Yamagishi}
\address{Mathematisch Instituut\\
Universiteit Utrecht\\ 
Budapestlaan 6\\
NL-3584CD Utrecht\\
The Netherlands}

\email{s.yamagishi@uu.nl}


\begin{abstract}
We study the density of rational points on a higher-dimensional orbifold
$(\PP^{n-1},\Delta)$ when
 $\Delta$ is a $\QQ$-divisor involving hyperplanes.
This allows us to address a question of Tanimoto about  whether the set of rational points on such an orbifold  constitutes a thin set.
Our approach relies on  the Hardy--Littlewood circle method 
to first study an asymptotic version of 
Waring's problem for mixed powers. In doing so we make crucial use of the
 recent resolution of the main conjecture in
Vinogradov's mean value theorem, due to Bourgain--Demeter--Guth and Wooley.
\end{abstract}

\maketitle

\thispagestyle{empty}
\setcounter{tocdepth}{1}
\tableofcontents

\section{Introduction}

This paper is about the arithmetic of rational points on higher-dimensional orbifolds,   in the spirit of Campana
\cite{campana}.
We shall be concerned  with orbifolds $(\PP^{n-1},\Delta)$, where $\Delta$ is a $\QQ$-divisor that takes the shape
$$
\Delta=\sum_{i=0}^r\left(1-\frac{1}{m_i}\right) D_i,
$$
for irreducible divisors  $D_0,\dots, D_r$  on $\PP^{n-1}$ and integers  $m_0,\dots,m_r\geq 2$.
The arithmetic of {\em Campana-points} on orbifolds interpolates between the theory of rational and integral points on classical algebraic varieties, thereby opening up a new field of enquiry.

 The orbifold
$(\PP^{n-1},\Delta)$ is smooth if the divisor $\sum_{i=0}^r D_i$ is  strict normal crossings and it is said to be log-Fano if
$-K_{\PP^{n-1},\Delta}$ is ample,
where
$K_{\PP^{n-1},\Delta}=K_{\PP^{n-1}}+\Delta$.
Forthcoming work  of Pieropan, Smeets,
Tanimoto and V\'arilly-Alvarado \cite{c-manin} introduces the notion of 
Campana-points on higher-dimensional orbifolds and 
studies their distribution  on vector group compactifications.
Motivated by  the Manin conjecture for rational points on Fano varieties \cite{fmt}, it is very natural to ask
what one can say about the density  of Campana-points of bounded height on smooth log-Fano  orbifolds $(\PP^{n-1},\Delta)$.
We shall address this in the special case that $D_0,\dots,D_r$ form a set of
distinct  hyperplanes in $\PP^{n-1}$, all defined over $\QQ$.
Then $(\PP^{n-1},\Delta)$ is log-Fano 
precisely when
$$
n-(r+1)+\sum_{i=0}^r\frac{1}{m_i}>0.
$$
Since $m_i\geq 2$ this forces us to have $r\leq 2(n-1).$
It turns out that the analysis is  easy when $r\leq n-1$ and so the first challenging case is when $r=n$, in which case the condition for being log-Fano is
$$
\sum_{i=0}^n\frac{1}{m_i}>1.
$$

We shall take
$$
D_i=\begin{cases}
\{x_i=0\} & \text{ if $0\leq i\leq n-1$,}\\
\{c_0x_0+\dots+c_{n-1}x_{n-1}=0\} & \text{ if $i=n$},
\end{cases}
$$
for a fixed choice of non-zero integers  $c_0,\dots,c_{n-1}$.
We let
$$
\Delta=\sum_{i=0}^{n}\left(1-\frac{1}{m_i}\right) D_i,
$$
for given integers $m_i\geq 2$.   The Campana-points in $(\PP^{n-1},\Delta)$
are defined to be  the rational points
$(x_0:\dots:x_{n-1})\in \PP^{n-1}(\QQ)$, represented by primitive integer vectors
$(x_0,\dots,x_{n-1})\in \ZZ_{\neq 0}^{n}$ for which
$x_i$ is $m_i$-full for $0\leq i\leq n-1$ and
$c_0x_0+\dots+c_{n-1}x_{n-1}$ is $m_n$-full.
Here, we recall that a non-zero integer $x$ is said to be $m$-full if $p^m\mid x$ whenever there is a prime $p$ such that  $p\mid  x$.

We attach the usual exponential height function $H: \PP^{n-1}(\QQ)\to \RR$, given by
$H(x_0:\dots: x_{n-1})=\max_{0\leq i\leq n-1}|x_i|$ if $(x_0,\dots,x_{n-1})\in \ZZ^{n}$ is primitive.
The counting function of interest to us here is then
\begin{equation}\label{eq:N(B)}
N(\PP^{n-1},\Delta;B)=\frac{1}{2}\#\left\{\x\in \ZZ_{\neq 0}^{n+1}:
\begin{array}{l}
\gcd(x_0,\dots,x_{n-1})=1\\
|\x|\leq B,
\text{ $x_i$ is $m_i$-full $\forall ~i$}\\
c_0x_0+\dots+c_{n-1}x_{n-1}=x_n
\end{array}{}
\right\},
\end{equation}
where $\x=(x_0,\dots,x_n)$ and $|\x|=\max_{0\leq i\leq n}|x_i|.$
In the special case   $m_0=\dots=m_n=2$, work of Van Valckenborgh \cite{KVV} establishes an asymptotic formula for
$N(\PP^{n-1},\Delta;B)$
 for all   $n\geq 4$. Drawing inspiration from this, we have the following generalisation.

\begin{thm}\label{t:1}
Assume that
 $m_0,\dots,m_{n}\geq 2$ and
\begin{eqnarray}
\label{cond1}
\sum_{\substack{0 \leq i \leq  n\\ i\neq j}} \frac{1}{m_i(m_i+1)} \geq 1
\end{eqnarray}
for some $j\in \{0,\dots,n\}$.
Then there exist  constants $c\geq 0$
and  $\eta>0$
such that
$$
N(\PP^{n-1},\Delta;B)= c B^{\sum_{i=0}^n \frac{1}{m_i}-1}
+
O\left(B^{\sum_{i=0}^n \frac{1}{m_i}-1-\eta}\right).
$$
\end{thm}

The implied constant in this estimate is allowed to depend on $m_0,\dots,m_n, n$ and 
$c_0,\dots, c_{n-1}$, a convention that we shall adopt for all of the implied constants in this paper.   There is an explicit expression for the leading constant $c$ in
\eqref{eq:c1''} and \eqref{eq:airport}, as a convergent sum of local densities. It can be shown that $c>0$ if the underlying equations admit suitable non-singular solutions everywhere locally.
In Theorem \ref{t:1}  the  exponent of $B$ is equal to
$$
a=a(L,\Delta)=\inf\left\{\ell\in \RR: \ell [L]+[K_{\PP^{n-1},\Delta} ]\in C_{\text{eff}}(\PP^{n-1})\right\},
$$
where $[L] $ is the class of a hyperplane section in $\PP^{n-1}$. Moreover, the  exponent of $\log B$ is $b-1$, where
$b=b(L,\Delta)$ is
 the codimension of the minimal face of $\partial C_\text{eff}(\PP^{n-1})$ that contains
$a [L]+[K_{\PP^{n-1},\Delta}]$.  
When $m_0=\dots=m_n=2$ and $n\geq 4$ the
work of Van Valckenborgh \cite{KVV} shows that
the
asymptotic formula for $N(\PP^{n-1},\Delta;B)$
follows the same pattern,  with $a=\frac{n-1}{2}$ and $b=1$.
However, some caution must be exercised when asking
to what extent other  orbifolds conform to this behaviour, as the following result shows.

\begin{thm}\label{t:mid}
Let $n=3$ and
 $m_0=m_1=m_2=m_3=2$. Then
$$
N(\PP^{2},\Delta;B)\gg B\log B.
$$
\end{thm}


On the other hand, when $n=2$ we expect the counting function to
satisfy an asymptotic formula with associated constants $a=\frac{1}{2}$ and $b=1$. In fact,
 Browning and Van Valckenborgh \cite{bvv} have produced an explicit constant $c>0$ such that
$N(\PP^1,\Delta;B)\geq c(1+o(1)) B^{\frac{1}{2}}$
when $m_0=m_1=m_2=2$.

\medskip

Let $X$ be an integral variety over $\QQ$.
Recall from Serre
\cite[\S3.1]{Ser08}
that
 a {\em thin set} is a set  contained in a finite
union of thin subsets of $X(\QQ)$ of type $I$ and $II$. Here,
 a {\em type $I$ thin subset}  is a set of the form $Z(\QQ) \subset X(\QQ)$, where $Z$ is a
proper closed subvariety, and a {\em type $II$ thin subset} is a set of the form
$f(Y(\QQ))$, where $f : Y \to X$ is a generically finite dominant morphism with
$\dim Y=\dim X$,  $\deg f \geq 2$ and $Y$ geometrically integral.
It follows from
work of  Cohen  \cite{cohen} (as further expounded by    Serre
 \cite[Thm.~13.3]{Ser97}) that the set
 $\PP^{n-1}(\QQ)$  is not thin.
 At the workshop
``Rational and integral points via analytic and geometric methods'' in  Oaxaca (May 27th--June 1st, 2018),
 Sho Tanimoto raised the  question of whether the same is true for the
 set of  Campana-points.  Our next goal is to provide some partial evidence in favour of this.

Associated to any type II thin subset $\Omega\subset \PP^{n-1}(\QQ)$ coming from a  morphism $Y \to \PP^{n-1}$ of degree $d\geq 2$
is a degree $d$  extension of function fields $\QQ(Y)/\QQ(t_1,\dots,t_{n-1})$. We let 
$\QQ(Y )^{\mathrm{Gal}}$ be the Galois closure of 
$\QQ(Y)$ over the function field $\QQ(t_1,\dots,t_{n-1})$ of $\PP^{n-1}$ and we let 
$\QQ_\Omega\subset \QQ(Y )^{\mathrm{Gal}}$ be the largest subfield that is algebraic over $\QQ$.  Finally we let 
$\mathsf{P}_\Omega$ be the set of rational primes that split completely in
$\QQ_\Omega$. It follows from the Chebotarev density theorem that 
$\mathsf{P}_\Omega$ has density $[\QQ_\Omega:\QQ]^{-1}$ in the set of primes, since $\QQ_\Omega/\QQ$ is Galois.
Next, let
\begin{equation}\label{eq:Q}
\mathsf{Q}_{\mathbf{m}}=\left\{\text{$p$ prime}:
\begin{array}{r}
\mathrm{lcm}
\left(\gcd(m_0,p-1),\dots,\gcd(m_n,p-1)\right)
\\
=\prod_{0\leq i\leq n}\gcd(m_i,p-1)
\end{array}
\right\},
\end{equation}
for any $\mathbf{m}=(m_0,\dots,m_n)\in \ZZ_{\geq 2}^{n+1}$.
 The following result provides an explicit  condition on the possible   thin sets that the Campana-points in $(\PP^{n-1},\Delta)$ are allowed to lie in.

\begin{thm}\label{t:2}
Assume that $m_0,\dots,m_n\geq 2$ and  \eqref{cond1} holds.
Let $\Omega\subset \bigcup_i\Omega_i$ be a thin set
where each  $\Omega_i\subset \PP^{n-1}(\QQ)$ is a thin subset of type I or II.
Assume that
\begin{equation}\label{eq:density}
\liminf_{x\to \infty}
\frac{\#\{p\in \mathsf{P}_{\Omega_i}\cap \mathsf{Q}_{\mathbf{m}}: p\leq x\}}{\pi(x)} >0
\end{equation}
whenever $\Omega_i$ is type II.
Then
$$N_{\Omega}(\PP^{n-1},\Delta;B) =o_\Omega(B^{\sum_{i=0}^n \frac{1}{m_i}-1}),$$
where
$N_{\Omega}(\PP^{n-1},\Delta;B)$ is defined as in \eqref{eq:N(B)}, but with the extra constraint that
$(x_0:\dots:x_{n-1})\in \Omega$.
\end{thm}

Assuming that \eqref{cond1} holds, we may combine this result with Theorem \ref{t:1} to deduce that the    Campana-points in $(\PP^{n-1},\Delta)$ are not contained in
	any thin set satisfying the hypotheses of the theorem.
The statement of this  result is rather disappointing at first glance, but in fact the conclusion is false when the condition
\eqref{eq:density} is dropped.
To see this, take  $m_0=\dots=m_n=3$ and $n\geq 12$. Then
$\sum_{i=0}^n \frac{1}{m_i}-1=\frac{n-2}{3}$ and \eqref{cond1} holds in Theorem \ref{t:1}.
Consider the thin set
$\Omega_0\subset \PP^{n-1}(\QQ)$  that arises from the
 morphism
$$
Z \to \PP^{n-1}, \quad
(x_0:\dots: x_{n}) \mapsto (x_0: \dots: x_{n-1}),
$$
where $Z\subset \PP^{n}$ is the cubic hypersurface 
$x_0^3+\dots+x_{n-1}^3=x_n^3$.
Then
the counting
function
$N_{\Omega_0}(\PP^{n-1},\Delta;B)$ has exact order $B^{\frac{n-2}{3}}$
for sufficiently large $n$. However,
\eqref{eq:density} fails in this case. Indeed,
$\mathsf{Q}_{\mathbf{m}}$ is the set of primes $p\not \equiv 1 \bmod{3}$, whereas
$\mathsf{P}_{\Omega_0}$ is the set of primes
$p \equiv 1 \bmod{3}$, since $\QQ_{\Omega_0}=\QQ(\sqrt{-3})$.
This shows that it is hard to approach Tanimoto's question in full generality through counting arguments alone.

The hypothesis \eqref{eq:density} is a little awkward to work with.
If one restricts attention to $\mathbf{m}$ such that
\begin{equation}\label{eq:gcd-m}
\gcd(m_j, m_{j'}) = 1 \quad  \text{ for $0 \leq j < j' \leq n$},
\end{equation}
then $\mathsf{Q}_\mathbf{m}$ is equal to the full set of rational primes. Moreover,  it follows from  Chebotarev's density theorem that $\mathsf{P}_{\Omega}$ has density $[\QQ_{\Omega}:\QQ]^{-1}$, for any type II thin subset $\Omega$. Thus the conditions of Theorem \ref{t:2} are met for any thin set. However, the assumption \eqref{cond1} is
 too stringent to cope with a sequence of integers $\geq 2$ that  satisfy \eqref{eq:gcd-m}.

\medskip

Our proof of Theorems \ref{t:1}--\ref{t:2} relies on an explicit description of $m$-full integers $x$.
For such integers  every exponent of a prime appearing in the prime factorisation of $x$ can be written  $k m + (m + r)$,  for integers $k \geq 0$ and $0 \leq r < m$.
Thus any non-zero   $m$-full integer $x$ can be written uniquely in the form
\begin{equation}\label{eq:sign}
x = \textnormal{sign}(x) \  u^{m} \prod_{r = 1}^{m-1} v_r^{m + r},
\end{equation}
for  $u, v_1, \ldots, v_{m - 1} \in \mathbb{N}$, such that  $\mu^2(v_r) =1$ for $1 \leq r \leq m-1$ and $\gcd(v_r, v_{r'}) = 1$ for $1 \leq r < r' \leq m-1$.

It may be  instructive to  illustrate this notation by  discussing the special case $m_0=\dots=m_n=2$, in which case
Campana-points in $(\PP^{n-1},\Delta)$ correspond to vectors $\mathbf{u},\mathbf{v}\in \NN^{n+1}$ and $\boldsymbol{\epsilon}\in \{\pm 1\}^{n+1}$ with each $v_j$ square-free,
 for which
$$
\epsilon_0c_0u_0^2v_0^3+\dots+\epsilon_{n-1}c_{n-1}u_{n-1}^2v_{n-1}^3=\epsilon_n u_n^2v_n^3.
$$
When $n=3$ we can clearly find vectors $\v\in \NN^4$ with square-free components and $\boldsymbol{\epsilon}\in \{\pm 1\}^4$ in such a way that
$$
-\epsilon_0\dots\epsilon_3 c_0c_1c_{2} v_0^3\dots v_3^3 =\square.
$$
Fixing such a choice and applying \cite[Thm.~7]{HB-crelle} to estimate the residual number of $\u\in \NN^{4}$  that lie on the split quadric, with $u_j\leq \sqrt{B/v_j^3}$ for $0\leq j\leq 3$,
we readily deduce that $N(\PP^{2},\Delta;B)\gg B\log B$, as claimed in Theorem \ref{t:mid}

Returning now to the  case of general $m_0,\dots,m_n\geq 2$,
we  summarise the structure of the paper.
Under the representation \eqref{eq:sign} it follows that  Campana-points on $(\PP^{n-1},\Delta)$ can be viewed through the lens of Waring's problem for mixed exponents.
Given its proximity to  Vinogradov's mean value theorem,
this is an area that has received a radical new injection of ideas at the hands of Wooley \cite{W0,W,W1} and Bourgain, Demeter and Guth \cite{BDG}.
Based on this, in \S \ref{s:HL} we shall give a completely general treatment
of the counting function associated to  suitably constrained integer solutions to the Diophantine  equation
$$
\sum_{0 \leq j \leq n} c_{j} \gamma_{j}  u_{j}^{m_j} = N,
$$
for given $N\in \ZZ$ and  non-zero $c_j,\gamma_j\in \ZZ$,
in which the vectors $\u$ are asked to lie in a  congruence class modulo $H$.
In this part of the argument we shall need to retain uniformity in the coefficients $\gamma_j$ and in the modulus $H$. It is here that the condition \eqref{cond1} arises.
The resulting  asymptotic formula is recorded in Theorem \ref{t:M}.
In \S \ref{s:manin}
we shall use Theorem \ref{t:M}  to establish the version of orbifold Manin that we have presented in Theorem \ref{t:1}. One of the chief difficulties in this part of the argument comes from dealing
with the coprimality conditions implicit in
the counting function $N(\PP^{n-1},\Delta;B)$.
Next,  in \S \ref{s:thin} we shall combine Theorem \ref{t:M} with information about the size of thin sets modulo $p$ (for many primes $p$) to
 tackle Theorem \ref{t:2}.

Finally, when $H=1$ and $c_j=\gamma_j=1$ for all $0\leq j\leq n$,   it is easy to derive from Theorem
\ref{t:M} an  asymptotic formula for the mixed Waring problem. The following result  may be of independent interest.

\begin{thm}\label{t:MW}
Assume that $m_0,\dots,m_n\geq 2$ and  \eqref{cond1} holds.
Let $R(N)$ denote the number of representations of a positive integer $N$ as
$$
N=x_0^{m_0}+\dots + x_n^{m_n}.
$$
Then there exists  $\eta>0$ such that
$$
R(N)=\frac{ \prod_{i=0}^n\Gamma(1+\frac{1}{m_i})}{\Gamma(\sum_{i=0}^n \frac{1}{m_i})} \mathfrak{S}(N) N^{\sum_{i=0}^n \frac{1}{m_i}-1} +O(N^{\sum_{i=0}^n \frac{1}{m_i}-1-\eta}),
$$
where $\mathfrak{S}(N)$ is given by \eqref{eq:eve}.
\end{thm}

There is relatively little in the literature concerning  asymptotic formulae for $R(N)$ for  mixed exponents. The best result is due to Br\"udern \cite{JB} who obtains an asymptotic formula for $R(N)$ when $m_0=m_1=2$, under some further conditions on the exponents, the most demanding of which is that
  $$
  \sum_{i=2}^n \frac{1}{m_i}>1.
  $$
Theorem \ref{t:MW} is not competitive with this, although it does not suffer from the  defect that $2$ must appear twice among the list of exponents.
It remains an interesting open challenge to prove an asymptotic formula for $R(N)$ for any value of  $n$, when $m_i=2+i$ for $0\leq i\leq n$.

When $m = m_0 = \dots = m_n$, which is the traditional  setting of  Waring's problem,
the  condition in \eqref{cond1} reduces to
$
n  \geq m^2+m. 
$
This shows that our approach is not completely optimal
in the equal exponent situation, since as  explained in
\cite[Cor.~14.7]{W1}, we know that 
$n\geq m^2-m+O(\sqrt{m})$
variables  suffice to get an asymptotic
formula in Waring's problem. It seems likely that by combining methods  developed by Wooley in \cite{WI} and \cite[\S 14]{W1}, one can   recover this loss. (The authors are grateful to Professor Wooley for this remark.)

\subsection*{Acknowledgements}
While working on this paper
the   authors were  both supported  by EPSRC grant \texttt{EP/P026710/1}, 
and the second author received additional support 
from  the NWO Veni Grant \texttt{016.Veni.192.047}.  Thanks are due to
 Marta Pieropan, Arne Smeets and Sho Tanimoto  for useful conversations related to this topic.

\section{The Hardy--Littlewood circle method}\label{s:HL}

We shall assume without loss of generality that
$2 \leq m_0\leq m_1 \leq \dots \leq m_n$. Our assumption
\eqref{cond1} translates into
\begin{equation}
\label{cond2}
\sum_{\substack{0 \leq j <  n}} \frac{1}{m_j(m_j+1)} \geq 1.
\end{equation}
In what follows it will be convenient to set
\begin{eqnarray}
\label{defnGamma}
\Gamma = \sum_{j=0}^{n} \frac{1}{m_j} - 1.
\end{eqnarray}
Let $N\in \ZZ$ and let $\mathbf{c}=(c_0,\dots,c_n)\in (\mathbb{Z} \backslash \{ 0 \})^{n+1}$. Let $H \in \mathbb{N}$, $\boldsymbol{\gamma} \in \mathbb{N}^{n+1}$ and let
$\mathbf{h} \in \{ 0, 1, \ldots, H-1 \}^{n+1}$.
The main results in this paper are founded on
an analysis of the counting function
\begin{equation}
\label{counting Mcgamma}
M_{\boldsymbol{\mathbf{c} ; \gamma} }(B ; \mathbf{h},{H};N)
=  \# \left\{  \mathbf{u} \in \mathbb{N}^{n+1} :
\begin{array}{l}
\gamma_{j} u_{j}^{m_j} \leq B, \text{ for $0 \leq j \leq n$}\\
\mathbf{u} \equiv \mathbf{h} \bmod{H}\\
\sum_{0 \leq j \leq n} c_{j} \gamma_{j}  u_{j}^{m_j} = N
\end{array}
 \right\}.
\end{equation}
We shall view $\c$ as being fixed, once and for all, but
$\boldsymbol{\gamma}$ can grow and so we will need all of our estimates to depend explicitly on
it.  In Theorems \ref{t:1} and \ref{t:2} we shall take $N=0$ and $c_n=-1$, whereas in Theorem \ref{t:MW} we take $H=1$, $c_j=\gamma_j=1$ and $B =N$.

Let
$$
B_{j} = (B/\gamma_{j})^{1/m_j},
$$
and
$$
S_{j} (\alpha) = \sum_{ \substack{ 1 \leq  u \leq B_{j} \\ u \equiv h_{j} \bmod{H} } } e\left(\alpha  c_{j} \gamma_{j} u^{m_j} \right),
$$
for $0\leq j \leq n$. Then we may write
\begin{eqnarray}
\label{orthog1}
M_{ \mathbf{c} ;  \boldsymbol{\gamma} }(B ; \mathbf{h},{H};N ) = \int_0^1 \mathcal{S}_{\boldsymbol{\gamma}} (\alpha) \rd \alpha,
\end{eqnarray}
where
$$
\mathcal{S}_{\boldsymbol{\gamma}}(\alpha) =e(-\alpha N) \prod_{j=0}^{n} S_{j} (\alpha).
$$
Note that we may freely assume that  $\gamma_j \leq B$ for  $0 \leq j \leq n$, since otherwise
$M_{\boldsymbol{\mathbf{c} ; \gamma} }(B ; \mathbf{h},{H};N ) = 0$.
Let $\delta$ be such that
\begin{eqnarray}
\label{deltass1}
0<\delta < \frac{1}{(2n+5) m_n (m_n+1)}.
\end{eqnarray}
We define the \textit{major arcs} $\mathfrak{M}$ to be
$$
\mathfrak{M} = \bigcup_{ \substack{ 0 \leq a \leq q \leq B^{\delta} \\ \gcd(a,q) = 1 }  } \mathfrak{M}(a,q),
$$
where
$$
\mathfrak{M}(a,q) = \{ \alpha \in [0,1) : |\alpha - a/q| < B^{-1+\delta} \}.
$$
We define the \textit{minor arcs} to be  $\mathfrak{m} = [0,1) \backslash \mathfrak{M}$.

\subsection{Contribution from the major arcs}

In the standard way we shall need to show that on the major arcs our exponential sums can be approximated by integrals, with acceptable error.  The following result is a straightforward adaptation of familiar facts.

\begin{lem}
\label{lem3.1}
Let $h,H \in \mathbb{N} \cup \{0\}$ with  $0 \leq h < H$. Let $X \geq 1$.
Let $a \in \mathbb{Z}$, $q \in \mathbb{N}$, $\beta \in \mathbb{R}$ and $\alpha = a/q  + \beta$.
Then
\begin{align*}
\sum_{ \substack{ 1 \leq u \leq X  \\ u \equiv h \bmod{H}  }} e(\alpha u^{m})
=~&\frac{X}{qH} \ \sum_{k=0}^{q-1} e\left(\frac{a (Hk + h)^m }{q}\right) \  \int_{0}^{1} e( \beta X^m z^{m}) \rd z\\
&\quad+ O \left( q + q X^m |\beta| \right).
\end{align*}
\end{lem}
\begin{proof}
Let $X' = (X-h)/H$. If $X'<q$ then the absolute value of the left hand side is trivially bounded by $q+1$, and so we may proceed under the assumption that $X'
 \geq q$. We write
\begin{align*}
\sum_{ \substack{ 1 \leq u \leq X  \\ u \equiv h \bmod{H}  }} e(\alpha u^{m})
&=
\sum_{ \substack{ 0  < x \leq  X'  }} e(\alpha (H x + h )^{m}) + O(1)
\\
&=
\hspace{-0.2cm}
\sum_{k=0}^{q - 1} e \left( \frac{a (H k + h )^{m}}{q} \right)
\hspace{-0.2cm}
\sum_{ \substack{ 0  <  x \leq  X'  \\  x \equiv k \bmod{q}  }} e(\beta (H x + h )^{m}) + O(1).
\end{align*}
The inner sum is
$$
\sum_{ \substack{ 0  <  x \leq  X'  \\  x \equiv k \bmod{q}  }} e(\beta (H x + h )^{m})
=\sum_{ \substack{ -k/q  <  y \leq  (X'-k)/q }} e(\beta (qH y + h+Hk )^{m}).
$$
An application of  the Euler--Maclaurin summation  formula
to this  sum now yields the result.
\end{proof}

Now let $\alpha = a/q + \beta \in \mathfrak{M}(a,q)$.
We apply Lemma \ref{lem3.1} with $X=B_{j}$,
and $\alpha$ (resp.~ $a$) replaced by
$\alpha c_{j} \gamma_{j}$ (resp. $a c_{j}  \gamma_{j}$).
Thus
$\alpha c_{j} \gamma_{j} - a c_{j}  \gamma_{j}/q = \beta c_{j} \gamma_{j}$ and
$$
q + q B_{j}^{m_j} | \beta c_{j} \gamma_{j} | \ll q + q B |\beta | \ll B^{2 \delta}.
$$
Put
$$
\mathfrak{J}_{ \mathbf{c} }(L) =  \int_{-L}^Le(-\lambda N/B)
\ \prod_{j=0}^n \
\int_{0}^{1} e\left(\lambda c_{j}  z^{m_j}  \right)  \rd z \rd \lambda
$$
and set $\mathfrak{S}_{ \mathbf{c} ; {\boldsymbol{\gamma}}  }(L ;  \mathbf{h},{H};N)$ to be
$$
  \sum_{q \leq L} \frac{1}{q^{n+1}}
\sum_{\substack{ 0 \leq a < q \\   \gcd(a,q) = 1 } } e\left(-\frac{aN}{q}\right)
\ \prod_{j=0}^n \
 \sum_{ 0 \leq k < q  }
e\left(\frac{a}{q}  c_{j} \gamma_{j}   (H k + h_{j})^{m_j}  \right),
$$
for any $L>1$.
Then it follows from Lemma \ref{lem3.1} that
\begin{equation}
\begin{split}
\label{major arc of Sv'}
\int_{\mathfrak{M}} \mathcal{S}_{\boldsymbol{\gamma}}(\alpha)  \rd\alpha
=~&
\mathfrak{S}_{\mathbf{c};\boldsymbol{\gamma} }(B^{\delta} ;  \mathbf{h},{H};N)   \mathfrak{J}_{ \mathbf{c} }(B^{\delta})
\frac{  \prod_{j=0}^n  B_{j}}{H^{n+1} B}
+ O\left(E_1(\boldsymbol{\gamma}; H)\right),
\end{split}\end{equation}
where
$$
E_1(\boldsymbol{\gamma}; H)
=  B^{-1+\delta }  \sum_{q \leq B^{\delta}}
q  \sum_{y = 0}^{n} (B^{2 \delta})^{n+1-y} \max_{ j_1 < \dots < j_y } \prod_{i=1}^y  \left( \frac{B_{j_i}}{H} + 1 \right).
$$
Taking $H\geq 1$ and observing that $B_j\geq 1$ for all $0\leq j\leq n$ we see that
\begin{align*}
\max_{ j_1 < \dots < j_y } \prod_{i=1}^y  \left( \frac{B_{j_i}}{H} + 1 \right)
&\ll  \left(\frac{1}{B_0}+\dots+\frac{1}{B_n}\right)\prod_{j=0}^n B_j.
\end{align*}
On executing the sum over $q$ we therefore conclude that
\begin{equation}\label{eq:E1}
E_1(\boldsymbol{\gamma}; H)
\ll
\frac{\prod_{j=0}^n B_j}{B}
 \left(\frac{1}{B_0}+\dots+\frac{1}{B_n}\right)B^{(2n+5)\delta}.
\end{equation}
It remains to analyse the  terms $\mathfrak{S}_{\mathbf{c};\boldsymbol{\gamma} }(B^{\delta} ;  \mathbf{h},{H};N)$ and $\mathfrak{J}_{\mathbf{c}}(B^{\delta})$ .

Beginning with the singular series, it follows from \cite[Theorem 7.1]{V} that
\begin{equation}\label{eq:vaughan}
\Big{|} \sum_{0 \leq k < q} e\left(\frac{x (Hk + h)^m}{q} \right) \Big{|}
\ll
\gcd(x,q)^{1/m} H q^{1 - 1/m + \varepsilon}
\end{equation}
for any $\varepsilon > 0$.
Therefore
\begin{align*}
\Big|\sum_{ X < q \leq Y } \frac{1}{q^{n+1}}
\sum_{ \substack{ 0 \leq a < q  \\ \gcd(a,q) = 1} }
e\left(-\frac{aN}{q}\right)
 \prod_{j=0}^n \  \sum_{ 0 \leq k < q  }
&e\left(\frac{a}{q} c_{j} \gamma_{j} (H k + h_{j})^{m_j}  \right) \Big|\\
&\ll
E_2(\boldsymbol{\gamma}; H ; X, Y),
\end{align*}
where
$$
E_2(\boldsymbol{\gamma}; H ; X, Y)
= H^{n+1} \sum_{ X <  q \leq Y } q^{-\Gamma+ \varepsilon }
\prod_{j=0}^n
\gcd(\gamma_{j}, q)^{ 1/{m_j} }.
$$
Put
\begin{equation}\label{eq:E2}
E_2(\boldsymbol{\gamma}; H )
= H^{n+1} \sum_{ q=1 }^\infty q^{1-\Gamma+ \varepsilon }
\prod_{j=0}^n
\gcd(\gamma_{j}, q)^{ 1/{m_j} }.
\end{equation}
Clearly
$E_2(\boldsymbol{\gamma};  H ; B^{\delta}, \infty)\leq B^{-\delta} E_2(\boldsymbol{\gamma}; H )$ and
$E_2(\boldsymbol{\gamma};  H ; 0, \infty)\leq E_2(\boldsymbol{\gamma}; H )$.

In view of  \eqref{cond2}, we have
\begin{equation}
\label{bdd on rm/m}
\sum_{j=0}^n \frac{1}{m_j} \geq 3.
\end{equation}
Let us define
\begin{equation}\label{eq:SS}
\begin{split}
\mathfrak{S}_{ \mathbf{c} ; {\boldsymbol{\gamma}}  }(\mathbf{h},{H};N) = \sum_{q=1}^\infty \frac{1}{q^{n+1}}
&\sum_{\substack{ 0 \leq a < q \\   \gcd(a,q) = 1 } }
e\left(-\frac{aN}{q}\right)
\\
&\times
\prod_{j=0}^n \
 \sum_{ 0 \leq k < q  }
e\left(\frac{a}{q}  c_{j} \gamma_{j}   (H k + h_{j})^{m_j}  \right).
\end{split}
\end{equation}
This is absolutely convergent, since \eqref{eq:E2} and (\ref{bdd on rm/m}) yield
$$
 \mathfrak{S}_{ \mathbf{c};\boldsymbol{\gamma} }(\mathbf{h},{H};N)\ll E_2(\boldsymbol{\gamma} ; H ; 0, \infty)
\ll H^{n+1} \prod_{j=0}^n \gamma_{j}^{1/m_j}.
$$
Moreover,
\begin{eqnarray}
\label{ssest1}
\mathfrak{S}_{\mathbf{c};\boldsymbol{\gamma} }( B^{\delta} ; \mathbf{h},{H};N )  = \mathfrak{S}_{ \mathbf{c};\boldsymbol{\gamma} }(\mathbf{h},{H};N) +  O \left( B^{-\delta}E_2(\boldsymbol{\gamma}; H ) \right).
\end{eqnarray}

Turning to the singular integral, it follows
 from \cite[Lemma 2.8]{V} that
$$
 \int_{0}^{1} e\left(\lambda c_{j} z^{m_j}  \right)  \rd z  \ll \min \{ 1, | \lambda |^{-1/m_j} \}.
$$
Thus,  in view of (\ref{bdd on rm/m}), we deduce that
\begin{align*}
\int_{|\lambda| \geq B^{\delta} }
\prod_{j=0}^n \left|
\int_{0}^{1} e\left(\lambda c_{j}  z^{m_j}  \right)  \rd z  \right|  \rd \lambda
&\ll
\int_{|\lambda| \geq B^{\delta} } |\lambda|^{ - \sum_{ j=0 }^n \frac{1}{m_j}}  \rd \lambda
\ll
B^{- \delta \Gamma}.
\end{align*}
Hence
$$
\mathfrak{J}_{ \mathbf{c}} = \int_{-\infty}^\infty
e(-\lambda N/B)
\prod_{j=0}^n
\int_{0}^{1} e\left(\lambda c_{j}  z^{m_j}  \right)  \rd z    \rd \lambda
$$
is well-defined, and we have
\begin{eqnarray}
\label{siest1}
|\mathfrak{J}_{\mathbf{c}} - \mathfrak{J}_{ \mathbf{c}}(B^{\delta}) |  \ll  B^{ - \delta \Gamma}
\leq B^{-\delta}.
\end{eqnarray}

We are now ready to conclude our treatment of the major arcs.
Note that
$$
\frac{ \prod_{j=0}^n B_{j}}{H^{n+1} B}=\frac{B^\Gamma}{H^{n+1} \prod_{j=0}^n \gamma_j^{1/m_j}}.
$$
On combining
 (\ref{major arc of Sv'}), (\ref{ssest1}) and (\ref{siest1}), we therefore obtain the following result.

 \begin{lem}\label{lem:major}
 Assume that \eqref{bdd on rm/m} holds. Then
\begin{align*}
\int_{\mathfrak{M}} \mathcal{S}_{ \boldsymbol{\gamma} }(\alpha)  \rd\alpha
=~&
\frac{\mathfrak{S}_{ \mathbf{c};\boldsymbol{\gamma}  } ( \mathbf{h},{H};N)
 \mathfrak{J}_{  \mathbf{c} } }{
 H^{n+1} \prod_{j=0}^n \gamma_j^{1/m_j}
  }  B^{\Gamma}
+ O\left(E_1(\boldsymbol{\gamma}; {H})
+
 \frac{B^{ \Gamma- \delta   } E_2(\boldsymbol{\gamma};  H )}{
 \prod_{j=0}^n \gamma_j^{1/m_j}}\right).
\end{align*}
\end{lem}

\subsection{Contribution from the minor arcs}

According to work of Wooley
\cite[Eq.~(1.8)]{W1},
the main conjecture in Vinogradov's mean value theorem asserts that for each
$\varepsilon > 0$ and $t, k \in \mathbb{N}$, one has
\begin{equation}
\label{VMT3}
\int_{[0,1)^k} \Big{|} \sum_{1 \leq x \leq X}
\hspace{-0.2cm}
e(\alpha_k x^k + \alpha_{k-1}x^{k-1} + \dots + \alpha_1 x)  \Big{|}^{2t} \rd \boldsymbol{\alpha} \ll X^{t + \varepsilon} +  X^{2t - \frac{k(k+1)}{2}}.
\end{equation}
This result was established recently by Bourgain, Demeter and Guth \cite{BDG} using $\ell^2$-decoupling and also by Wooley \cite{W, W1} using efficient congruencing.
The following mean value estimate is a straightforward consequence of their work.

\begin{lem}
\label{improved Hua in residue class}
Let $k \in \mathbb{N}$ and let $s$ be a real number satisfying $s \geq k(k+1)$.
Let $A, H \in \mathbb{Z} \backslash \{0\}$ and $h \in \mathbb{Z}$.
Then we have
$$
\int_0^1 \Big{|} \sum_{1 \leq x \leq X} e(\alpha A (Hx + h)^k) \Big{|}^{s} \rd\alpha \ll X^{s - k},
$$
where the implied constant does not depend on  $A$, $H$ or $h$.
\end{lem}
\begin{proof}
Let $2t$ be the largest even integer such that $2t\leq s$. Then it follows that $t \geq k(k+1)/2$.
By a trivial estimate and by considering the underlying equations of the following integrals via the orthogonality relation, we deduce that
\begin{align*}
\int_{0}^1 \left| \sum_{1 \leq x \leq X} e(\alpha A (Hx + h)^k )  \right|^{s}
\rd \alpha
&\leq X^{s - 2t} \int_{0}^1 \left| \sum_{1 \leq x \leq X} e(\alpha A (Hx + h)^k )  \right|^{2t}
\rd \alpha\\
&=X^{s - 2t}
\sum_{\substack{
\mathbf{n}=(n_1,\dots,n_{k-1})\in \ZZ^{k-1}\\
-t X^{j} < n_j < t X^{j} }} I(\mathbf{n}),
\end{align*}
where
$$
I(\mathbf{n})
= \hspace{-0.1cm}
\int_{[0,1)^k}  \Big{|} \sum_{1 \leq x \leq X}
\hspace{-0.2cm}
e(\alpha_k A (Hx + h)^k + \alpha_{k-1}x^{k-1} + \dots + \alpha_1 x)  \Big{|}^{2t}
e \left( -\mathbf{n}.\boldsymbol{\alpha}'\right)
\rd\boldsymbol{\alpha},
$$
where
$
\boldsymbol{\alpha}=(\alpha_k,\dots,\alpha_1)$ and
$
\boldsymbol{\alpha}'=(\alpha_{k-1},\dots,\alpha_1)$.
Summing trivially over $\mathbf{n}$, the right hand side of our estimate
is
\begin{align*}
&\ll
X^{\frac{k(k-1)}{2}} \int_{[0,1)^k} \left| \sum_{1 \leq x \leq X} e(\alpha_k A (Hx + h)^k + \alpha_{k-1}x^{k-1} + \dots + \alpha_1 x)  \right|^{2t}
\rd \boldsymbol{\alpha}
\\
&= X^{\frac{k(k-1)}{2}} \int_{[0,1)^k} \left| \sum_{1 \leq x \leq X} e(\alpha_k x^k + \alpha_{k-1}x^{k-1} + \dots + \alpha_1 x)  \right|^{2t} \rd\boldsymbol{\alpha},
\end{align*}
with  the last equality an immediate consequence of  considering the underlying equations of the integrals. An application of  (\ref{VMT3}) now yields our result.
\end{proof}

We also require the following Weyl type estimate, which is another consequence of the recent work on Vinogradov's mean value theorem.
We omit the proof since it is  obtained by invoking  the main conjecture (\ref{VMT3}) in the proof of \cite[Theorem 1.5]{W0}.
\begin{lem}
\label{Weyl type ineq}
Let $k \geq 2$ and let $\alpha_k, \ldots, \alpha_1 \in \mathbb{R}$. Suppose there exists $a \in \mathbb{Z}$ and $q \in \mathbb{N}$ with $\gcd(a,q)=1$ satisfying $| \alpha_k -  a/q | \leq q^{-2}$ and $q \leq X^k$.
Let
$$
0 \leq \sigma(k) \leq \frac{1}{k(k-1)}.
$$
Then
$$
 \sum_{1 \leq x \leq X} e(\alpha_k x^k + \alpha_{k-1}x^{k-1} + \dots + \alpha_1 x)  \ll X^{1 + \varepsilon} (q^{-1} + X^{-1} + q X^{-k})^{\sigma(k)},
$$
for any $\ve>0$.
\end{lem}

Using this result we obtain the following bound
for the exponential sum on the minor arcs.

\begin{lem}
\label{lem5.3}
Let $\ve>0$. Then 
$$
\sup_{ \alpha \in \mathfrak{m} } |S_{n} (\alpha)| \ll B^{\frac{1}{m_n} - 
\frac{\delta}{m_n(m_n+1)}
  + \varepsilon}  \gamma_{n}^{  - \frac{1}{m_n+1}}.
$$
\end{lem}
\begin{proof}
It will be convenient throughout the proof to write 
$$
\sigma(m_n) = \frac{1}{m_n(m_n+1)}.
$$
Let $\alpha \in \mathfrak{m}$ and let  $\beta = \alpha c_{n} \gamma_{n} H^{m_n}$.
We put
$$
\widetilde{B} = \min \left\{ 2 B^{1 - \delta},   \frac{B}{|c_{n}| \gamma_{n}  H^{m_n} 2^{m_n} } \right\}.
$$
When $\widetilde{B} \leq 1$ it is clear that
$S_{n} (\alpha) \ll 1$. Since $\gamma_{n} \leq B$ we have
$$
B^{1/m_n - \delta \sigma(m_n)  + \varepsilon}  \gamma_{n}^{  - 1/m_n+\sigma(m_n) }
\geq
1 \gg S_n(\alpha)
$$
in this case.
Thus we may suppose that $\widetilde{B} > 1$.

By Dirichlet's theorem on Diophantine approximation we know there exist
$b \in \mathbb{Z}$ and $1 \leq r \leq \widetilde{B}$ such that $\gcd(b,r)=1$ and
$$
|\beta - b/r| \leq 1/(r \widetilde{B}) \leq 1/r^2.
$$
Note that $b\neq 0$ since $\alpha\in \mathfrak{m}$.
For simplicity let us write $A = c_{n} \gamma_{n}  H^{m_n}$. We claim that  $b A > 0$. But if  $b A < 0$ then
$$|\beta-b/r|=|\alpha A-b/r|=\alpha|A|+|b|/r>1/r,
$$
since $\alpha>0$, which is a contradiction.
This establishes the claim. Let   $A' = A / \gcd(A, b)$ and $b' = b / \gcd(A, b)$.

Let $X = (B_{n} - h_{n})/H$.
First suppose $B_{n}/2H > X$. Then $B_{n} < 2 h_{n} < 2 H.$
In this case we clearly have
$ S_{n} (\alpha) \ll 1$, which is satisfactory.
Thus we suppose $B_{n}/2H \leq X$.
In this case $r \leq \widetilde{B} \leq X^{m_n}$ and  Lemma \ref{Weyl type ineq} yields
\begin{equation}\begin{split}
\label{ineq 2}
 S_{n} (\alpha)  &=   \sum_{1 \leq x \leq X} e \left( \beta \left( x + \frac{h}{H} \right)^{m_n} \right)  + O(1)\\
&\ll X^{1 + \varepsilon} (r^{-1} + X^{-1} + rX^{-m_n})^{\sigma(m_n)} ,
\end{split}
\end{equation}
for any $\ve>0$.
Next, we note that
$$
\frac{1}{|A| r \widetilde{B} } \geq \frac{1}{|A|} \  \Big{|}\beta - \frac{b}{r} \Big{|}  = \Big{|}\alpha - \frac{b}{r A} \Big{|}
= \Big{|}\alpha - \frac{|b'|}{ r |A'| } \Big{|}.
$$
If $2B^{1 - \delta} \leq B/|c_{n} \gamma_{n} H^{m_n} 2^{m_n} |$ it follows that
$$
\Big{|} \alpha -  \frac{|b'|}{r |A'|} \Big{|} \leq \frac{1}{ |A| r \widetilde{B}} \leq \frac{1}{ \widetilde{B}} < B^{- 1+\delta}.
$$
On the other hand, if $2B^{1 - \delta} > B/|c_{n} \gamma_{n} H^{m_n} 2^{m_n} |$ then
$$
\Big{|}\alpha - \frac{|b'|}{ r |A'| } \Big{|} \leq \frac{1}{|A| r \widetilde{B} } \leq \frac{ |c_{n}| \gamma_{n} H^{m_n} 2^{m_n} }{|A| B  } \ll \frac{1}{B}.
$$
We now verify that $1 \leq |b'| \leq r |A'|$.
We've already seen that $|b'|\geq 1$, so we suppose that $|b'| > r |A'|$.
Since $\alpha \in [0,1)$ we have
$$
\frac{1}{r |A'|} \leq  \Big{|}\alpha - \frac{|b'|}{ r |A'| } \Big{|} \leq \frac{1}{ |A| r \widetilde{B}},
$$
whence
$
1 < \widetilde{B} \leq |A'|/|A| \leq 1.
$
This is a contradiction, so that we do indeed have  $1 \leq |b'| \leq r |A'|$. We also have $\gcd(r|A'|, |b'|) = 1$.
Finally, $\alpha \in \mathfrak{M}$ if $r|A'| \leq B^{\delta}$ and  $B$ is sufficiently large,  which is a contradiction. Therefore $r |A'| > B^{\delta}$ and
 (\ref{ineq 2}) becomes
\begin{align*}
 S_{n} (\alpha)   &\ll \frac{B^{1/m_n+\varepsilon}}{|A|^{1/m_n}} (|A'|/B^{\delta} + (B/|A|)^{-1/m_n} + |A|/B^{\delta})^{\sigma(m_n) }
\\
&\ll
 \frac{B^{1/m_n+\varepsilon}}{|A|^{1/m_n}} ( (B/|A|)^{-1/m_n} + |A|/B^{\delta})^{\sigma(m_n)}
\\
&\ll
B^{1/m_n - \delta \sigma(m_n)   + \varepsilon }  \   |A|^{\sigma(m_n)-1/m_n}.
\end{align*}
This completes the proof of the lemma, since 
$\sigma(m_n)-1/m_n=-1/(m_n+1)$.
\end{proof}

We now have the tools in place to establish the following bound for the minor arc contribiution.

\begin{lem}\label{lem:minor}
Assume that \eqref{cond2} holds and let $\ve>0$. Then
$$
\int_{\mathfrak{m}} \Big{|} \prod_{j=0}^{n} S_{j} (\alpha) \Big{|}  \rd\alpha
\ll
B^{\Gamma  - \frac{\delta}{m_n(m_n + 1)}  + \varepsilon }
\prod_{j=0}^{n} \gamma_{j}^{ - 1/(m_j+1)}.
$$
\end{lem}

\begin{proof}
Let $\ell_n=m_n(m_n+1)$ and let
 $\ell_{0},\dots,\ell_{n-1}>0$ be such that
$$
\sum_{0 \leq j < n} \frac{1}{\ell_{j}}  =  1.
$$
In the light of
 (\ref{cond2}) we can assume that  $\ell_{j} \geq m_j(m_j+1)$
 for all $0\leq j\leq n-1$.
It now follows from  H\"{o}lder's inequality and Lemma \ref{improved Hua in residue class} that
\begin{align*}
\int_{\mathfrak{m}} \left|  \prod_{j=0}^{n} S_{j} (\alpha) \right|  \rd\alpha
&\leq
\sup_{\alpha \in \mathfrak{m} } |S_{n} (\alpha)| \cdot \int_0^1 \left|  \prod_{j=0}^{n-1} S_{j} (\alpha) \right| \rd\alpha
\\
&\leq
\sup_{\alpha \in \mathfrak{m} } |S_{n} (\alpha)| \cdot \prod_{j=0}^{n - 1} \left( \int_0^1 |S_{j} (\alpha) |^{\ell_{j}} \rd\alpha \right)^{1/ \ell_{j}}
\\
&\ll
B^{\varepsilon} \cdot  \sup_{\alpha \in \mathfrak{m} } |S_{n} (\alpha)| \cdot
\prod_{j=0}^{n - 1} \left( \frac{B}{ \gamma_{j} }  \right)^{ \frac{\ell_{j} - m_j}{m_j  \ell_{j}}},
\end{align*}
since $H\geq 1$ and $\gamma_j\leq B$ for all $0\leq j\leq n-1$.
We apply Lemma \ref{lem5.3} to estimate $S_n(\alpha)$.
The statement of the lemma follows on simplifying the final expression and observing
that
$$
-\frac{1}{m_j}+\frac{1}{\ell_j}\leq -\frac{1}{m_j}+\frac{1}{m_j(m_j+1)}=-\frac{1}{m_j+1},
$$
for all $0\leq j\leq n-1$.
\end{proof}

\subsection{Final estimate}

We may now bring together Lemmas \ref{lem:major} and \ref{lem:minor}
in (\ref{orthog1}), in order to record the following estimate for the counting function
\eqref{counting Mcgamma}.

\begin{thm}
\label{t:M}
Assume that
$2\leq m_0\leq \dots \leq m_n$ and
 \eqref{cond2} holds.
 Let $\delta$ satisfy \eqref{deltass1}
  and let $\ve>0$.  Then
\begin{align*}
M_{\mathbf{c};\boldsymbol{\gamma}  } (B; \mathbf{h},{H};N )
\hspace{-0.1cm}
=~&
\hspace{-0.1cm}
\frac{\mathfrak{S}_{ \mathbf{c};\boldsymbol{\gamma}  } ( \mathbf{h},{H};N)
 \mathfrak{J}_{  \mathbf{c} } }{
 H^{n+1} \prod_{j=0}^n \gamma_j^{1/m_j}
  }  B^{\Gamma}
+ O\left(E_1(\boldsymbol{\gamma}; {H})
+
 \frac{B^{ \Gamma- \delta   } E_2(\boldsymbol{\gamma};  H )}{
 \prod_{j=0}^n \gamma_j^{1/m_j}}\right)\\
 &+
O \left( B^{\Gamma - \frac{\delta}{m_n(m_n+1)}  + \varepsilon }
\prod_{j=0}^{n} \gamma_{j}^{ - 1/(m_j+1)}  \right),
\end{align*}
where $E_1$ and $E_2$ are given by \eqref{eq:E1} and \eqref{eq:E2}, respectively.
\end{thm}

We end this section by indicating how this implies Theorem \ref{t:MW}, for which we
observe that $M_{\mathbf{c};\boldsymbol{\gamma}  } (B; \mathbf{h},{H};N )=R(N)$ when
$H=1$, $B=N$ and  $c_j=\gamma_j=1$ for $0\leq j\leq n$. The error term is clearly in the desired shape and recourse to \eqref{eq:SS}
shows that
$\mathfrak{S}_{ \mathbf{c};\boldsymbol{\gamma}  } ( \mathbf{h},{H};N)=\mathfrak{S}(N)$, with
\begin{equation}\label{eq:eve}
\mathfrak{S}(N)
= \sum_{q=1}^\infty \frac{1}{q^{n+1}}
\sum_{\substack{ 0 \leq a < q \\   \gcd(a,q) = 1 } }
e\left(-\frac{aN}{q}\right)
\prod_{j=0}^n \
 \sum_{ 0 \leq k < q  }
e\left(\frac{a}{q}  k^{m_j}  \right).
\end{equation}
Finally, standard arguments yield
 $$
\mathfrak{J}_{ \mathbf{c}} = \int_{-\infty}^\infty
e(-\lambda )
\prod_{j=0}^n
\int_{0}^{1} e\left(\lambda  z^{m_j}  \right)  \rd z    \rd \lambda =
\frac{ \prod_{i=0}^n\Gamma(1+\frac{1}{m_i})}{\Gamma(\sum_{i=0}^n \frac{1}{m_i})}.
$$
This therefore completes the proof of Theorem \ref{t:MW}.

\section{Orbifold Manin: proof of Theorem \ref{t:1}} \label{s:manin}

We now turn to the task of proving an asymptotic formula for the counting function
$N(\PP^{n-1},\Delta;B)$ in
Theorem  \ref{t:1}. We shall assume without loss of generality that
$2\leq m_0\leq \dots\leq m_{n}$, so that  \eqref{cond1} implies  \eqref{cond2}.
The  counting function  can be written
$$
N(\PP^{n-1},\Delta;B)=\frac{1}{2}\#\left\{\x\in \ZZ_{\neq 0}^{n+1}:
\begin{array}{l}
\gcd(x_0,\dots,x_{n})=1\\
|\x|\leq B,
\text{ $x_i$ is $m_i$-full $\forall ~0\leq i\leq n$}\\
c_0x_0+\dots+c_{n-1}x_{n-1}+c_nx_n=0
\end{array}{}
\right\},
$$
where we henceforth follow the convention that $c_n=-1$.
In view of \eqref{eq:sign}, we may write
$$
N(\PP^{n-1},\Delta;B)=\frac{1}{2}\#\left\{\x\in \ZZ_{\neq 0}^{n+1}:
\begin{array}{l}
\gcd(x_0,\dots,x_{n})=1, |\x|\leq B\\
x_{j} = \pm u_{j}^{m_j} \prod_{r=1}^{m_j-1} v_{j,r}^{m_j+r} ~\forall ~0\leq j\leq n\\
  \mu^2(v_{j,r}) = 1,~ \gcd(v_{j,r}, v_{j,r'}) = 1\\
c_0x_0+\dots+c_nx_n=0
\end{array}{}
\right\}.
$$

Suppose that we are given vectors
$\mathbf{s}$ and $\mathbf{t}$ with coordinates
$s_j\in \NN$
and $t_{j,r}\in \NN$
for $0\leq j\leq n$ and $
1 \leq r \leq m_j-1$.  It will be convenient to introduce the set
\begin{align*}
\mathcal{N}_{\mathbf{c}} (B; \mathbf{s}, \mathbf{t})
&=
\left\{
\x\in (\NN\cap [1,B])^{n+1}:
\begin{array}{l}
x_{j} = u_{j}^{m_j} \prod_{r=1}^{m_j-1} v_{j,r}^{m_j+r} ~\forall ~0\leq j\leq n\\
  \mu^2(v_{j,r}) = 1, ~\gcd(v_{j,r}, v_{j,r'}) = 1\\
c_0x_0+\dots+c_nx_n=0\\
\text{$s_{j} \mid  u_{j}$ and $ t_{j, r} \mid  v_{j,r}$} ~\forall ~j,r
\end{array}{}
\right\}.
\end{align*}
Given $\boldsymbol{\epsilon} \in \{ \pm 1 \}^{n+1}$ let $\boldsymbol{\epsilon} \mathbf{c}$ denote the vector with coordinates $\epsilon_{j} c_{j}$.
Then
\begin{equation}\label{eq:doggy}
N(\PP^{n-1},\Delta;B)=\frac{1}{2}
 \sum_{\boldsymbol{\epsilon} \in \{ \pm 1 \}^{n+1} }\# (\mathcal{N}_{\boldsymbol{\epsilon} \mathbf{c}}(B; \mathbf{1}, \mathbf{1}) \cap \mathbb{Z}^{n+1}_{{\textnormal{prim}}} ),
\end{equation}
where $\1$ is the vector with all coordinates equal to $1$.

We need to develop an inclusion-exclusion argument to cope with the coprimality condition in this expression. To ease notation we replace $\boldsymbol{\epsilon} \mathbf{c}$ by
$ \mathbf{c}$.
Let $\mathbf{x} \in \mathcal{N}_{\mathbf{c}}(B; \mathbf{1}, \mathbf{1})$.
It is clear that $\gcd(x_0, \dots, x_n) > 1$ if and only if
there exists a prime $p$ and a subset $I \subseteq \{0,\dots,n\}$ for which $p\mid  u_{j}$ for all $j \in I$ and also
$p \mid  \prod_{r=1}^{m_j-1} v_{j,r}$ for all $j \not \in I$.
(Note that  $I$ is allowed to be the empty set
 here.)

 Let $\mathcal{G}$ denote the set of all possible vectors  $\mathbf{g}\in \NN^{n+1}$ with
$1\leq g_{j}\leq m_j-1$ for $0\leq j\leq n$.
Let $\mathsf{P} = \{ 2, 3, 5, \ldots  \}$ denote the set of primes and
let $\mathcal{R}$ be a non-empty finite collection of triples $(\mathbf{g}; p; I)$ where $\mathbf{g} \in \mathcal{G}$,  $p \in \mathsf{P}$ and (possibly empty) $I \subseteq \{0,\dots,n\}$. Let $\mathcal{R}(p)$ be the subset of $\mathcal{R}$ containing all the triples in $\mathcal{R}$  with prime $p$.
In what follows we adhere to common  convention and stipulate that a union over the empty set is  the empty set and a product over the empty set is $1$.
We let
$$
I( \mathcal{R}(p) ) =  \bigcup_{ \substack{  (\mathbf{g}; p ; I) \in \mathcal{R}(p)  }  } I
\  \   \textnormal{  and  }    \   \
J(\mathbf{g}; \mathcal{R}(p) ) = \bigcup_{ \substack{  (\mathbf{g}'; p; I) \in \mathcal{R}(p)    \\
  \mathbf{g}' =  \mathbf{g} }  } \{0,\dots,n\} \backslash I.
$$
Next,
we define $\mathbf{a}(\mathcal{R})$
to be the vector in $\mathbb{N}^{n+1}$ with coordinates
\begin{eqnarray}
\label{defn amj}
a_{j} =   \prod_{ \substack{ p \in \mathsf{P}  \\   j \in I( \mathcal{R}(p) ) }  } p,  
\qquad
   (0 \leq j \leq n),
\end{eqnarray}
and we define
$\mathbf{b}(\mathcal{R})$
to be the vector in $\mathbb{N}^{\sum_{j=0}^n (m_j - 1)}$ with coordinates
\begin{eqnarray}
\label{defn bmj}
b_{j,r} = \prod_{
\substack{
p \in \mathsf{P} \\
j \in J(\mathbf{g}; \mathcal{R}(p) )  \text{ for some } \mathbf{g} \in \mathcal{G}
\\ \text{satisfying } g_{j} = r } } p,
\qquad (0 \leq j \leq n, 1 \leq r \leq m_j - 1).
\end{eqnarray}
It is easy to see that
$(\mathbf{a}(\mathcal{R}),   \mathbf{b}(\mathcal{R})) \not = (\mathbf{1}, \mathbf{1})$ as soon as  $\mathcal{R} \not = \emptyset$.
Moreover,   when $\mathcal{R} = \{ ( \mathbf{g}; {p};I) \}$ we see that
$\mathcal{N}_{\mathbf{c}}(B ; \mathbf{a}( \mathbf{g}; {p};I),  \mathbf{b}( \mathbf{g}; {p};{I}  ) )$ is precisely the set of
 $\mathbf{x} \in \mathcal{N}_{\mathbf{c}} (B; \mathbf{1}, \mathbf{1})$ satisfying
$p\mid u_j$ for all $j \in I$ and $p\mid  v_{j, g_j}$ for all $j \not \in I$.
In particular, it is now clear that
\begin{equation}
\label{inclnexcln0}
\mathcal{N}_{\mathbf{c}}(B; \mathbf{1}, \mathbf{1}) \cap \mathbb{Z}^{n+1}_{\textnormal{prim}}
=
\mathcal{N}_{\mathbf{c}}(B; \mathbf{1}, \mathbf{1})
\backslash
\bigcup_{ \substack{ \mathbf{g} \in \mathcal{G} \\ p \in \mathsf{P}  \\  I \subseteq \{0,\dots,n\}  } }
\mathcal{N}_{\mathbf{c}} (B;  \mathbf{a}(  \mathbf{g}; {p};I  ), \mathbf{b} ( \mathbf{g}; {p}; I ) ).
\end{equation}
We proceed by establishing the following result.

\begin{lem}
\label{lem5.1}
Given any $\mathcal{R} \not = \emptyset$,  we have
$$
\mathcal{N}_{\mathbf{c}}(B ; \mathbf{a}(\mathcal{R}),   \mathbf{b}(\mathcal{R})   )  =
\bigcap_{(\mathbf{g}; p; I) \in \mathcal{R}} \mathcal{N}_{\mathbf{c}}(B ; \mathbf{a}( \mathbf{g}; {p};I),  \mathbf{b}( \mathbf{g}; {p};{I}  ) ).
$$
\end{lem}
\begin{proof}
Let $\mathbf{x}$ belong to the intersection on the right hand side. Then,
given any $(\mathbf{g}; p; I) \in \mathcal{R}$, we have
$p \mid u_{j}$ for all $j \in I$ and $p \mid v_{j,r}$ if $j \not \in  I$ and $r = g_{j}$, where $x_j = u_j^{m_j} \prod_{r=1}^{m_j-1} v^{m_j + r}_{j,r}$.
Therefore, $p \mid u_{j}$ for all $p$ such that $j \in I(\mathcal{R}(p))$
and $p \mid v_{j,r}$ for all $p$ such that
$$
j \in  \bigcup_{ \substack{  (\mathbf{g}; p; I) \in \mathcal{R}(p) \\  g_{j} = r}  }   \{0,\dots,n\} \backslash I.
$$
Thus (\ref{defn amj}) and (\ref{defn bmj}) imply that $a_{j} \mid u_{j}$ and  $b_{j,r} \mid v_{j,r}$, whence  it follows that $\mathbf{x} \in \mathcal{N}_{\mathbf{c}}(B ; \mathbf{a}(\mathcal{R}),   \mathbf{b}(\mathcal{R}) )$.
On the other hand, if $\mathbf{x} \in \mathcal{N}_{\mathbf{c}}(B ; \mathbf{a}(\mathcal{R}),   \mathbf{b}(\mathcal{R}) )$
then we may reverse the argument to deduce that $\x$ also belongs to the intersection of all the sets
$\mathcal{N}_{\mathbf{c}}(B ; \mathbf{a}( \mathbf{g}; {p};I),  \mathbf{b}( \mathbf{g}; {p};{I}  ) )$
for
$(\mathbf{g}; p; I) \in \mathcal{R}$.
This completes the proof of the lemma.
\end{proof}

Given vectors $\mathbf{s}$ and $\mathbf{t}$ composed from positive integers, let
$$
\varpi (\mathbf{s}, \mathbf{t}) = \sum_{k=1}^{\infty} (-1)^{k}  \  \# \{ \mathcal{R} : \# \mathcal{R} = k, (\mathbf{s}, \mathbf{t}) = (\mathbf{a}(\mathcal{R}), \mathbf{b}(\mathcal{R}) ) \}.
$$
Then, on combining  the inclusion-exclusion principle with Lemma \ref{lem5.1}, we obtain
\begin{align*}
\# \bigcup_{ \substack{ \mathbf{g} \in \mathcal{G} \\ p \in \mathsf{P}  \\  I \subseteq \{0,\dots,n\}  } }
\hspace{-0.3cm}
\mathcal{N}_{\mathbf{c}}(B;  \mathbf{a}( \mathbf{g}; {p};I), \mathbf{b} (\mathbf{g}; {p};{I}) )
&= - \sum_{k=1}^{\infty} (-1)^{k} \sum_{ \substack{ \# \mathcal{R} = k  }  } \# \mathcal{N}_{\mathbf{c}}(B; \mathbf{a}(\mathcal{R}), \mathbf{b}(\mathcal{R}) )\\
&=-
\sum_{ \substack{   (\mathbf{s}, \mathbf{t}) \not = (\mathbf{1}, \mathbf{1})   }  }
\varpi (\mathbf{s}, \mathbf{t} ) \cdot   \# \mathcal{N}_{\mathbf{c}}(B; \mathbf{s}, \mathbf{t}  ).
\end{align*}
Note that $\# \mathcal{N}_{\mathbf{c}}(B; \mathbf{a}(\mathcal{R}), \mathbf{b}(\mathcal{R}) ) = 0$ when $\# \mathcal{R}$ is sufficiently large with respect to $B$.
Bringing this together with
\eqref{inclnexcln0}, we conclude that
\begin{equation}\label{eq:i-e}
\#\mathcal{N}_{\mathbf{c}}(B; \mathbf{1}, \mathbf{1}) \cap \mathbb{Z}^{n+1}_{\textnormal{prim}}
=
\#\mathcal{N}_{\mathbf{c}}(B; \mathbf{1}, \mathbf{1})
+
\sum_{ \substack{   (\mathbf{s}, \mathbf{t}) \not = (\mathbf{1}, \mathbf{1})   }  }
\varpi (\mathbf{s}, \mathbf{t} ) \cdot   \# \mathcal{N}_{\mathbf{c}}(B; \mathbf{s}, \mathbf{t}  ).
\end{equation}
It remains to asymptotically estimate these quantities.

We  collect
together some properties of the function $\varpi(\mathbf{s}, \mathbf{t})$.

\begin{lem}
\label{lemma on omega}
Let $(\mathbf{s}, \mathbf{t}) \not = (\mathbf{1}, \mathbf{1})$ and  let $p \in \mathsf{P}$. We let $\mathbf{s}^{[p]}$ be the vector whose $j$th coordinate is $s_{j}^{[p]} = p^{\val_p(s_{j})}$ and $\mathbf{t}^{[p]}$ be the vector whose $(j, r)$th coordinate is $t_{j,r}^{[p]} = p^{\val_p(t_{j,r})}$.
Then  the following are true:
\begin{enumerate}
\item[(i)]
 $\varpi(\mathbf{s}, \mathbf{t}) = \prod_{p \in \mathsf{P},
 ( \mathbf{s}^{[p]}, \mathbf{t}^{[p]} )\neq (\mathbf{1},\mathbf{1})
 } \varpi( \mathbf{s}^{[p]}, \mathbf{t}^{[p]} )$;

\item[(ii)]
$\varpi(\mathbf{s}, \mathbf{t} ) = 0$ if one of the coordinates of $\mathbf{s}$ or $\mathbf{t}$ is divisible by $p^2$;
\item[(iii)]
$\varpi(\mathbf{s}, \mathbf{t} ) = 0$ if one of the coordinates of $\mathbf{s}$ or $\mathbf{t}$ is divisible by $p$, but
there exists $0 \leq j \leq n$ with $s^{[p]}_{j} = t^{[p]}_{j,1} = \dots = t^{[p]}_{j,m_j-1} = 1$; and

\item[(iv)]
 $ \varpi( \mathbf{s}^{[p]}, \mathbf{t}^{[p]}  )  \ll 1.$
\end{enumerate}

\end{lem}
\begin{proof}
It follows from the definitions of $\mathbf{a}(\mathcal{R})$ and $ \mathbf{b}(\mathcal{R})$ that
$$
\mathbf{a}(\mathcal{R}) = \prod_{p  \in \mathsf{P}} \mathbf{a}(\mathcal{R}(p) )  \  \  \  \text{  and  }  \  \  \  \mathbf{b}(\mathcal{R}) = \prod_{p \in \mathsf{P}} \mathbf{b}(\mathcal{R}(p)),
$$
where we define multiplication of vectors by multiplying the corresponding coordinates.
We clearly have $(\mathbf{s}, \mathbf{t}) = \prod_{p \in \mathsf{P}} (\mathbf{s}^{[p]}, \mathbf{t}^{[p]})$ and $\# \mathcal{R} = \sum_{p \in \mathsf{P} } \# \mathcal{R}(p)$.
Thus
\begin{align*}
\prod_{\substack{p \in \mathsf{P}\\
 ( \mathbf{s}^{[p]}, \mathbf{t}^{[p]} )\neq (\mathbf{1},\mathbf{1})
}} \varpi(\mathbf{s}^{[p]}, \mathbf{t}^{[p]})
&=  \prod_{p \in \mathsf{P}} \sum_{k'=1}^{\infty} (-1)^{k'}
T_p(k'),
\end{align*}
where
$$
T_p(k')
=  \# \left\{ \mathcal{R} \subseteq \mathcal{G} \times \{ p \} \times \{0,\dots,n\} :
\begin{array}{l}
\# \mathcal{R} = k'\\
(\mathbf{s}^{[p]}, \mathbf{t}^{[p]}) = (\mathbf{a}(\mathcal{R}), \mathbf{b}(\mathcal{R}) )
\end{array}
\right\}.
$$
It follows that
\begin{align*}
\prod_{\substack{p \in \mathsf{P}\\
 ( \mathbf{s}^{[p]}, \mathbf{t}^{[p]} )\neq (\mathbf{1},\mathbf{1})
}}  \varpi(\mathbf{s}^{[p]}, \mathbf{t}^{[p]})
&=
\sum_{k=1}^{\infty} (-1)^k   \sum_{ \substack{  \sum k_p = k  } } \ \prod_{\substack{ p \in \mathsf{P}  \\ k_p > 0 } } T_p(k_p)
=
\varpi (\mathbf{s}, \mathbf{t}),
\end{align*}
which thereby
 establishes (i).

To prove (ii) we note that it is not possible for $p^2$ to divide any coordinate of  $\mathbf{a}(\mathcal{R}(p) )$
or $\mathbf{b}(\mathcal{R}(p))$ for any prime $p$ and $\mathcal{R} \not = \emptyset$. Thus  $\varpi (\mathbf{s}^{[p]}, \mathbf{t}^{[p]})  = 0$
if one of the coordinates of $\mathbf{s}^{[p]}$ or $\mathbf{t}^{[p]}$ is divisible by $p^2$.

Next, to prove (iii)  let $(\mathbf{s}, \mathbf{t}) \not = (\mathbf{1}, \mathbf{1})$
and assume without loss of generality that
$p \mid s_1 t_{1,1} \dots t_{1, m_1-1}$ and $s^{[p]}_{2} = t^{[p]}_{2,1} = \dots =  t^{[p]}_{2,m_{2}-1} =1$.
Suppose there exists $\mathcal{R}$ such that $(\mathbf{s}, \mathbf{t}) = (\mathbf{a}(\mathcal{R}), \mathbf{b}(\mathcal{R}) )$.
Then we have  $\mathcal{R}(p) \not = \emptyset$, and also
$$
2 \in \{0,\dots,n\} =  I(\mathcal{R}(p))  \cup    \bigcup_{\mathbf{g} \in \mathcal{G}}  J(\mathbf{g}; \mathcal{R}(p)).
$$
If $2 \in I(\mathcal{R}(p))$ then $p\mid s_{2}$. On the other hand, if $2 \in J(\mathbf{g}; \mathcal{R}(p))$ then $p\mid t_{2, g_{2}}$. In either case we have a contradiction, whence $\varpi(\mathbf{s}, \mathbf{t} ) = 0$.

Finally, to prove (iv) we note
there are only $O(1)$ options for $\mathcal{R}(p)$ for any fixed $p \in \mathsf{P}$. It now follows from the definition that
$$
| \varpi( \mathbf{s}^{[p]}, \mathbf{t}^{[p]}  ) |
\leq  \sum_{k=1}^{\infty}  \# \{ \mathcal{R}(p) :  \# \mathcal{R}(p)  = k  \}
\notag
\ll 1,
$$
as required.
\end{proof}
Given
$(\mathbf{s}, \mathbf{t}) \not = (\mathbf{1}, \mathbf{1})$
and $\varepsilon > 0$,
it follows from Lemma \ref{lemma on omega} that
\begin{equation}
\label{bound on omega}
\varpi (\mathbf{s}, \mathbf{t}) \ll \prod_{j=0}^n s_{j}^{\varepsilon} \prod_{1 \leq r \leq m_j-1} t_{j, r}^{\varepsilon}.
\end{equation}
We begin by studying
$$
\sum_{\substack { ( \mathbf{s}, \mathbf{t} ) \not = (\mathbf{1}, \mathbf{1}) } }
\varpi(\mathbf{s}, \mathbf{t} ) \cdot \# \mathcal{N}_{\mathbf{c}}(B; \mathbf{s}, \mathbf{t} ).
$$
Let
$$
\gamma_{j} = s_{j}^{m_j}  \prod_{r=1}^{m_j - 1} t_{j, r}^{m_j +r}  {v}_{j,r}^{m_j + r}.
$$
Then
$$
\mathcal{N}_{ \mathbf{c}}(B; \mathbf{s}, \mathbf{t} )
=
\left\{  \mathbf{x} \in \mathbb{N}^{n+1} :
\begin{array}{l}
x_{j} = {u}_{j}^{m_j} \prod_{r=1}^{m_j-1} {v}_{j,r}^{m_j+r},~
  \gamma_{j} {u}_{j}^{m_j} \leq B\\
 \mu^2 \left( {v}_{j,r} t_{j,r} \right) = 1\\
  \gcd \left( {v}_{j,r}  t_{j,r},   {v}_{j,r'} t_{j,r'} \right) = 1\\
 \sum_{0 \leq j \leq n} c_{j} \gamma_{j} {u}_{j}^{m_j} = 0
\end{array}
 \right\},
$$
where the indices run over $0 \leq j \leq n$ and $1 \leq r < r' \leq m_j-1$.
For each $\mathbf{s}$ and $\mathbf{t}$ we let
$$
\sideset{}{^{(1)}}\sum_{\mathbf{v}}
$$
denote the sum over all $\mathbf{v}$ satisfying $\gamma_{j} \leq B$ ,
$\gcd \left( {v}_{j,r}  t_{j,r}, {v}_{j,r'}  t_{j,r'} \right) = 1$
and $ \mu^2 \left( {v}_{j,r}  t_{j,r} \right) = 1$.
(If there is no $\mathbf{v}$ which satisfies the above conditions
then the sum is considered to be $0$.)
We may now write
$$
\# \mathcal{N}_{ \mathbf{c} }(B; \mathbf{s}, \mathbf{t} ) =
\sideset{}{^{(1)}}\sum_{\mathbf{v} }
M_{\mathbf{c}; \boldsymbol{\gamma} }(B),
$$
where $M_{\mathbf{c}; \boldsymbol{\gamma} }(B )=M_{\mathbf{c}; \boldsymbol{\gamma} }(B ;  \mathbf{0},{1};0)$,  in the notation \eqref{counting Mcgamma}.
Guided by Lemma~\ref{lemma on omega},
we let
$$
\sideset{}{^{(2)}}\sum_{\mathbf{s}, \mathbf{t} }
$$
denote the sum over $(\mathbf{s},\mathbf{t}) \not = (\mathbf{1}, \mathbf{1})$ satisfying
$s_{j}^{m_j} \prod_{r=1}^{m_j-1} t_{j,r}^{m_j + r}\leq B$ 
and $\gcd \left( t_{j,r}, t_{j,r'} \right) = 1$,
together with the condition that
none of the coordinates of $\mathbf{s}$ or $\mathbf{t}$ is divisible by $p^2$ for any  prime $p$
and if
one of the coordinates of $\mathbf{s}$ or $\mathbf{t}$ is divisible by $p$ then
$p \mid  s_j t_{j,1} \dots t_{j,m_j-1}$  for all $0 \leq j \leq n$.

We want to apply
Theorem \ref{t:M} with $H=1$ and $N=0$.
Let $\delta>0$ satisfy \eqref{deltass1} and
let
$\mathfrak{S}_{ \mathbf{c};\boldsymbol{\gamma}  }
=
\mathfrak{S}_{ \mathbf{c};\boldsymbol{\gamma}  } ( { \mathbf{0}},{1};0)$.
Then, on appealing to
 Lemma \ref{lemma on omega} and (\ref{bound on omega}), we deduce that
\begin{equation}
\label{eq:S1}
\sum_{\substack { (\mathbf{s}, \mathbf{t}) \not = (\mathbf{1}, \mathbf{1})  } }
\varpi(\mathbf{s}, \mathbf{t} ) \cdot \# \mathcal{N}_{\mathbf{c}}(B; \mathbf{s}, \mathbf{t} )
=
M(B)+
O\left(B^{\Gamma+\ve}\sum_{i=1}^3 F_i(B)\right),
\end{equation}
for any $\ve>0$,
where
$$
M(B)=B^{\Gamma }
\sideset{}{^{(2)}}\sum_{\mathbf{s}, \mathbf{t} }
 \varpi (\mathbf{s}, \mathbf{t} )   \sideset{}{^{(1)}}\sum_{\mathbf{v}}
\mathfrak{S}_{ \mathbf{c};\boldsymbol{\gamma}  }   \mathfrak{J}_{\mathbf{c}  }
\prod_{j=0}^n \gamma_{j}^{-1/m_j} .
$$
Moreover,
in view of \eqref{eq:E1} and \eqref{eq:E2},
the error terms are given by
\begin{align*}
F_1(B)&=
B^{(2n+5)\delta}\sum_{k=0}^n B^{-1/m_k}
\sideset{}{^{(2)}}\sum_{\mathbf{s}, \mathbf{t} } \sideset{}{^{(1)}}\sum_{\mathbf{v}}
\prod_{\substack{j=0\\ j\neq k}}^n \gamma_j^{-1/m_j},
\\
F_2(B) &= B^{-\delta}\sideset{}{^{(2)}}\sum_{\mathbf{s}, \mathbf{t} } \sideset{}{^{(1)}}\sum_{\mathbf{v}}
 \sum_{ q=1 }^\infty q^{1-\Gamma+ \varepsilon }
\prod_{j=0}^n
\frac{\gcd(\gamma_{j}, q)^{ 1/{m_j} }}{\gamma_{j}^{1/m_j} },\\
F_3(B)
&=B^{ - \frac{\delta}{m_n(m_n+1)}}
\sideset{}{^{(2)}}\sum_{\mathbf{s}, \mathbf{t} } \sideset{}{^{(1)}}\sum_{\mathbf{v}}
\prod_{j=0}^{n} \gamma_{j}^{ -1/(m_j+1)}.
\end{align*}
We now need to estimate these three error terms.
In doing so it will be convenient to set
$w_{j} = v_{j,1}^{m_j+1} \dots v_{j, m_j-1}^{2m_j - 1}$
and  $\tau_{j} = s_{j}^{m_j} \prod_{r=1}^{m_j-1} t_{j,r}^{m_j + r}$.

Now for any $\tau \geq 1$, we have
\begin{align*}
\sum_{v_1^{m+1} \dots v_{m-1}^{2m-1} \leq B / \tau} 1
&\ll
\sum_{ v_2,\dots, v_{m-1}=1}^\infty  \left(\frac{B/ \tau}{v_2^{m + 2} \dots v_{m-1}^{2m-1}} \right)^{1/(m + 1)}
\ll \left(\frac{B}{\tau}\right)^{1/(m + 1)}.
\end{align*}
Similarly,
\begin{align*}
\sum_{v_1^{m+1} \dots v_{m-1}^{2m-1} \leq B / \tau}
 \left(\frac{1}{v_1^{m + 1} \dots v_{m-1}^{2m-1}} \right)^{1/m }
\ll 1.
\end{align*}
Using these estimates it follows that
\begin{align*}
F_1(B)
&\ll B^{(2n+5)\delta}\sum_{k=0}^n B^{-1/m_k}
\sideset{}{^{(2)}}\sum_{\mathbf{s}, \mathbf{t} }
\left(\frac{B}{\tau_k}\right)^{1/(m_k+1)}\prod_{\substack{j=0\\ j\neq k}}^n \tau_j^{-1/m_j}\\
&\ll
B^{-1/m_n(m_n+1)+(2n+5)\delta}
\sideset{}{^{(2)}}\sum_{\mathbf{s}, \mathbf{t} }
\prod_{\substack{j=0}}^n \tau_j^{-1/(m_j+1)},
\end{align*}
where we recall that 
 $\tau_{j} = s_{j}^{m_j} \prod_{r=1}^{m_j-1} t_{j,r}^{m_j + r}$.
Lemma \ref{lemma on omega} now yields
\begin{equation}\label{eq:st}
\sideset{}{^{(2)}} \sum_{\mathbf{s}, \mathbf{t}}
 \prod_{j=0}^n \tau_{j}^{-1/(m_j + 1)}
\leq
\prod_{\substack{ p }}  \left( 1 +  \prod_{j = 0}^n (2m_j-1) p^{-m_j/(m_j+1)} \right)
\ll 1,
\end{equation}
since $\sum_{j=0}^n m_j/(m_j + 1)>1$.
(Note that the factor $2m_j-1$ on the right hand side  comes from taking into account the $2m_j-1$ possibilities where the factor $p$ appears in $s_j$ or $t_j$.)
We have therefore shown that
$$
F_1(B)
\ll B^{-1/m_n(m_n+1)+(2n+5)\delta}.
$$

Turning to the estimation of  $F_2(B)$, we may write
 \begin{align*}
F_2(B)
&\leq
B^{-\delta}
\sum_{ q=1 }^\infty q^{1-\Gamma+ \varepsilon } f_1(q)f_2(q) ,
\end{align*}
where
\begin{align*}
f_1(q)&=
\sideset{}{^{(2)}}\sum_{\mathbf{s}, \mathbf{t}}
\prod_{j=0}^{n}  \left( \frac{\gcd(s_{j}^{m_j}  t_{j,1}^{m_j+1} \dots t_{j,m_j-1}^{2m_j-1}, q)}{ s_{j}^{m_j}  t_{j,1}^{m_j +1} \dots t_{j,m_j -1}^{2m_j-1} } \right)^{1/m_j} ,\\
f_2(q)&= \sum_{\substack{w_j\leq B\\ \mu^2(v_{j,r})=1}}  \prod_{j=0}^{n}  \frac{\gcd(w_{j},q)^{1/m_j}  }{w_{j}^{1/m_j}}.
\end{align*}
We first show that
\begin{equation}\label{eq:desk}
\sum_{ \substack{ x \leq  B^{1/(m + r)}    } }
\frac{\mu^2(x)\gcd(x^{m+r},q)^{1/m}} { x^{(m + r)/m}}
\ll q^\ve,
\end{equation}
if $r\geq 1$.
To see this we note that the left hand side is at most
$$
\sum_{d \mid q} d^{1/m }
\sum_{ \substack{ x \leq  B^{ 1/(m + r) }  \\  d\mid  x^{m+r} } }
\frac{\mu^2(x)}{x^{(m + r)/m}}.
$$
When $\mu^2(x)=1$,
any  $d\mid  x^{m+r}$ admits a factorisation
 $d = d_1d_2^2 \dots d_{m+r}^{m+r}$
 such that $d_1\dots d_{m+r}\mid x$,
  where $\mu^2(d_i)=1$ and $\gcd(d_i, d_j) = 1$ for $i \not = j$.
If we write $x = x' d_1 \dots d_{m+r}$, then
 this sum is
\begin{align*}
&\leq
\sum_{\substack{d \mid q \\ d = d_1 \dots d_{m+r}^{m+r}  }} \frac{(d_1 \dots d_{m+r}^{m+r})^{1/m}}{ (d_1 \dots d_{m+r})^{(m+r)/m} }
\sum_{ \substack{ x' \leq  B^{ 1/(m + r)}/(d_1 \dots d_{m+r})}}
\frac{1}{{x'^{(m+r)/m}}}.
\end{align*}
The inner $x'$-sum is 
absolutely convergent  since $r\geq 1$. 
The remaining sum over $d\mid q$ is $O(q^\ve)$ for any $\ve>0$, by the standard estimate for the divisor function.
This therefore establishes \eqref{eq:desk}.

An application of \eqref{eq:desk} immediately yields
\begin{equation}\label{eq:B}
f_2(q)\leq
 \prod_{j=0}^{n} \prod_{r=1}^{m_j-1}
\sum_{ \substack{ v_{j,r} \leq  B^{ 1/(m_j + r) }   } }
\frac{\mu^2(v_{j,r})\gcd(v_{j,r}^{m_j+r},q)^{1/m_j}} { v_{j,r}^{(m_j + r)/m_j}}\ll
q^\ve,
\end{equation}
for any $\ve>0$.
Next, let 
$$
f_{1,T}(q)=
\sideset{}{^{(2)}}\sum_{\substack{\mathbf{s}, \mathbf{t}\\ 
\max\{\tau_0,\dots, \tau_n\}\geq T}
}
\prod_{j=0}^{n} \tau_j^\ve \left( \frac{\gcd(\tau_j, q)}{\tau_j } \right)^{1/m_j}
$$
for any $\ve> 0$ and $T\geq 1$, where
  $\tau_{j} = s_{j}^{m_j} \prod_{r=1}^{m_j-1} t_{j,r}^{m_j + r}$.
In particular we have $f_1(q)\leq f_{1,1}(q)$. 
We claim that 
\begin{equation}\label{eq:nobacon}
f_{1,T}(q)\ll q^{6\ve(m_0+\dots+m_n)} T^{-\ve}
\end{equation}
for any sufficiently small $\ve>0$.
Once achieved, 
it will follow that
$$
F_2(B)
\ll
B^{-\delta },
$$
since (\ref{bdd on rm/m}) implies that
 $\Gamma-1>1$.

To check the claim we let $\mathcal{T}$ denote the set of vectors 
$(\tau_0,\dots,\tau_n)\in \NN^{n+1}$ with the property that for any prime $p$ we have 
$\val_p(\tau_j)\in \{0,m_j,\dots,3m_j-1\}$ and, furthermore,  if $p\mid \tau_0\dots\tau_n$ then 
$\val_p(\tau_j)>0$ for all $0\leq j\leq n$. 
Associated to any $(\tau_0,\dots,\tau_n)\in \mathcal{T}$
is a unique choice for $\mathbf{s}, \mathbf{t}$.
Thus we find that 
\begin{align*}
f_{1,T}(q)
&\ll \frac{1}{T^\ve}
\sum_{(\tau_0,\dots,\tau_n)\in \mathcal{T}}
\prod_{j=0}^{n} \tau_j^{2\ve} \left( \frac{\gcd(\tau_{j}, q)}{ \tau_j} \right)^{1/m_j}\\
&\ll \frac{1}{T^\ve}
\prod_p 
\left(1+
\prod_{j=0}^n
\sum_{m_j\leq \alpha_j \leq 3m_j-1}  p^{\min(\alpha_j,\val_p(q))/m_j -\alpha_j/m_j+2\ve \alpha_j}\right).
\end{align*}
When $p\nmid q$ the corresponding local factor takes the shape $$1+O(p^{-n-1+2\ve(m_0+\dots+m_n)}).$$
Alternatively, when $p\mid q$ the factor is $O(p^{6\ve(m_0+\dots+m_n)})$ 
Assuming that $\ve$ is sufficiently small this therefore concludes the proof of \eqref{eq:nobacon}.

Finally we must analyse
$$
F_3(B)
=B^{ - \frac{\delta}{m_n(m_n+1)}}
\sideset{}{^{(2)}}\sum_{\mathbf{s}, \mathbf{t} } \sideset{}{^{(1)}}\sum_{\mathbf{v}}
\prod_{j=0}^{n} \gamma_{j}^{ -1/(m_j+1)}.
$$
We note that
\begin{align*}
\sum_{v_1^{m+1} \dots v_{m-1}^{2m-1} \leq B}
 \left(\frac{1}{v_1^{m + 1} \dots v_{m-1}^{2m-1}} \right)^{1/(m+1) }
\ll \log B.
\end{align*}
Applying  \eqref{eq:st} to handle the resulting sum over $\s$ and $\t$ it easily follows that
$$
F_3(B)
\ll B^{ - \frac{\delta}{m_n(m_n+1)}+\ve},
$$
for any $\ve>0$.

We substitute our bounds for   the error terms
back into \eqref{eq:S1}. This yields
\begin{equation}
\label{eq:S1'}
\begin{split}
\sum_{\substack { (\mathbf{s}, \mathbf{t}) \not = (\mathbf{1}, \mathbf{1})  } }
&\varpi(\mathbf{s}, \mathbf{t} ) \cdot \# \mathcal{N}_{\mathbf{c}}(B; \mathbf{s}, \mathbf{t} )
\\&=
M(B)
+
O\left(B^{\Gamma+\ve}
\left\{
B^{ - \frac{1}{m_n(m_n+1)} +(2n+5)\delta}
+  B^{
- \frac{\delta}{m_n(m_n+1)}  } \right\}
\right).
\end{split}
\end{equation}

It remains to consider the case
$( \mathbf{s}, \mathbf{t} ) = (\mathbf{1}, \mathbf{1})$, for which we
rerun the above argument, with the special choice
$(\mathbf{s},\mathbf{t}) = (\mathbf{1}, \mathbf{1})$. The starting point is
$$
\# \mathcal{N}_{ \mathbf{c} }(B; \mathbf{1}, \mathbf{1} ) =
\sideset{}{^{(1)}}\sum_{\mathbf{v} }
M_{\mathbf{c}; \boldsymbol{\gamma} }(B),
$$
where now $\boldsymbol{\gamma}$ has components
$
\gamma_{j} = \prod_{r=1}^{m_j - 1}   {v}_{j,r}^{m_j + r}.
$
Tracing through the argument, this ultimately leads to the conclusion
\begin{equation}
\label{eq:S2'}
\begin{split}
 \# \mathcal{N}_{\mathbf{c}}(B; \mathbf{1}, \mathbf{1} )
=
\widetilde M(B)
+
O\left(B^{\Gamma+\ve}
\left\{
B^{ - \frac{1}{m_n(m_n+1)} +(2n+5)\delta}
+  B^{
- \frac{\delta}{m_n(m_n+1)}  }
 \right\}
\right),
\end{split}
\end{equation}
for any $\ve>0$,
where now
$$
\widetilde M(B)=B^{\Gamma }   \sideset{}{^{(1)}}\sum_{\mathbf{v}}
\mathfrak{S}_{ \mathbf{c};\boldsymbol{\gamma}  }   \mathfrak{J}_{\mathbf{c}  }
\prod_{j=0}^n \gamma_{j}^{-1/m_j} .
$$

We are now ready to complete the proof of Theorem \ref{t:1}.
Repeating the arguments used in
\eqref{eq:nobacon} during our  analysis of $F_2(B)$, it is easy to 
remove the constraint
$s_{j}^{m_j} \prod_{r=1}^{m_j-1} t_{j,r}^{m_j + r}\leq B$ from the 
 summation over $\mathbf{s}, \mathbf{t}$  in the main term $M(B)$. The total error in doing this is $O(B^{\Gamma-\eta_1})$, for some $\eta_1>0$. 
We proceed under the assumption that \eqref{cond2} holds and $\delta$ satisfies
\eqref{deltass1}.
We may  combine
\eqref{eq:doggy} and \eqref{eq:i-e}
with \eqref{eq:S1'} and
\eqref{eq:S2'} in order to   conclude that
$N(\PP^{n-1},\Delta;B)$ is
$$
c_B B^{\Gamma}+
O\left(B^{\Gamma+\ve}
\left\{
B^{ - \frac{1}{m_n(m_n+1)} +(2n+5)\delta}
+  B^{
- \frac{\delta}{m_n(m_n+1)}  }
+ B^{- \eta_1}
 \right\}
\right),
$$
for any $\ve>0$, where
$$
c_B=
\frac{1}{2}
 \sum_{\boldsymbol{\epsilon} \in \{ \pm 1 \}^{n+1} }
\mathfrak{J}_{\boldsymbol{\epsilon} \mathbf{c} }
\left(
\sum_{\substack { (\mathbf{s}, \mathbf{t})= (\mathbf{1},\mathbf{1})}}
+
\sum_{\substack { (\mathbf{s}, \mathbf{t})\neq (\mathbf{1},\mathbf{1})}}
\varpi(\mathbf{s},\mathbf{t})\right)
 \sideset{}{^{(1)}}\sum_{\mathbf{v}}
\frac{
\mathfrak{S}_{ \boldsymbol{\epsilon} \mathbf{c};\boldsymbol{\gamma}  }}{ \prod_{j=0}^n \gamma_{j}^{1/m_j}}.
$$
The error term is of the shape claimed in    Theorem \ref{t:1} and so it remains to analyse the quantity $c_B$.

The dependence on $B$ in the factor  $c_B$ arises
from the definition of the sum $\sum^{(1)}$.
A straightforward repetition of our arguments above suffice to show that
$$
c_B=c+O(B^{-\eta_2})
$$
for some $\eta_2>0$,
where  $c$ is the constant that is defined as in $c_B$, but with the summation conditions
$\gamma_j\leq B$ removed from $\sum^{(1)}$, for $0\leq j\leq n$.
This shows that 
$N(\PP^{n-1},\Delta;B)=
c B^{\Gamma}+O(B^{\Gamma-\eta})$ for an appropriate $\eta>0$, as claimed in Theorem \ref{t:1}. To go further, we adopt
 the notation $\s^{\m}  \w  = (s_0^{m_0} w_0, \dots, s_n^{m_n} w_n)$, where
we recall that $w_{j} = v_{j,1}^{m_j+1} \dots v_{j, m_j-1}^{2m_j - 1}$ for $0\leq j\leq n$.
Changing the order of summation, we may write
\begin{equation}
\label{eq:c1''}
c =
\frac{1}{2}
 \sum_{\boldsymbol{\epsilon} \in \{ \pm 1 \}^{n+1} }
\hspace{-0.3cm}
\mathfrak{J}_{\boldsymbol{\epsilon} \mathbf{c} }
\hspace{-0.3cm}
\sum_{\mathbf{v} \in \mathbb{N}^{\sum_{j=0}^n (m_j-1) } }
\hspace{-0.4cm}\frac{\prod_{j=0}^n \mu^2(v_{j,1} \dots v_{j, m_j-1} ) }{\prod_{j=0}^n  w_{j}^{1/m_j}}
\sum_{\substack { \mathbf{s}, \mathbf{t} \\ \mathbf{t}| \v }}
\varpi(\mathbf{s},\mathbf{t})
\frac{
\mathfrak{S}_{ \boldsymbol{\epsilon} \mathbf{c}; \s^{\m}  \w  }}{ \prod_{j=0}^n  s_j},
\end{equation}
with the understanding that $\varpi(\mathbf{1}, \mathbf{1}) = 1$
and  $\t \mid \v$ means $t_{j,r} \mid v_{j,r}$ for all $j$ and $r$. 
We claim that 
\begin{equation}\label{eq:airport}
\sum_{\substack { \mathbf{s}, \mathbf{t} \\ \mathbf{t}| \v }}
\varpi(\mathbf{s},\mathbf{t})
\frac{
\mathfrak{S}_{ \boldsymbol{\epsilon} \mathbf{c}; \s^{\m}  \w  }}{ \prod_{j=0}^n  s_j}=
\prod_{p} \left( \lim_{T \rightarrow \infty}  \frac{\mathcal{M}_{ \boldsymbol{\epsilon}, T} (\v, p) }{p^{nT}} \right),
\end{equation}
where
\begin{equation}\label{eq:alt}
\mathcal{M}_{ \boldsymbol{\epsilon}, T}(\v, p) = \# \left\{ \k \bmod p^T :  
\begin{array}{l}
\sum_{j=0}^n \epsilon_jc_j w_j k_j^{m_j} \equiv 0 \bmod p^T\\
 \exists j \textnormal{ such that }
p \nmid k_j v_{j,1} \dots v_{j, m_j-1}  
\end{array}
\right\}.
\end{equation}
This will complete our analysis of the leading constant $c$ in Theorem \ref{t:1}.

To check the claim we put 
   $c'_j = \epsilon_j c_j$ for $0\leq j\leq n$.
It follows from  \eqref{eq:SS} and multiplicativity that 
$$
\frac{\mathfrak{S}_{ \mathbf{c}'; \s^{\m}  \w }}{\prod_{j=0}^n s_j }
= \prod_{p} \frac{1}{\prod_{j=0}^n s_j^{[p]} }  \left( 1 + \sum_{t=1}^{\infty} \mathcal{B}_{\s^{\m}  \w}(p^t) \right),
$$
where
$$
\mathcal{B}_{\mathbf{s^m}\w}(p^t) = \frac{1}{p^{t(n+1)} }
\sum_{\substack{ 0 \leq a < p^t \\   \gcd(a,p^t) = 1 } }
\prod_{j=0}^n \  \sum_{ 0 \leq k < p^t  }
e \left(\frac{a}{p^t} c'_{j} {s_j^{m_j}w_{j}}  k^{m_j}  \right).
$$
Letting
$$
N
(p^T) = \# \left\{ \k \bmod p^T : \sum_{j=0}^n c'_j s_j^{m_j} w_j k_j^{m_j} \equiv 0 \bmod p^T \right\},
$$
we deduce that 
$$
\frac{\mathfrak{S}_{ \mathbf{c}'; \s^{\m}  \w }}{\prod_{j=0}^n s_j }
= \prod_{ \substack{ p \\ p \nmid s_0 \dots s_n } } \left(  \lim_{T \rightarrow \infty} \frac{N(p^T)}{p^{nT}} \right)
\prod_{ \substack{ p \\ p |  s_0 \dots s_n } } \left(  \lim_{T \rightarrow \infty} \frac{N(p^T)}{ p^{nT} \prod_{j=0}^n s_j^{[p]} } \right).
$$
Next, we put 
$$
\mathcal{X}_{p, T}(\s, \t)  =  \left\{ \k \bmod p^T :
\begin{array}{l}
 \sum_{j=0}^{n} c_j' w_j k_j^{m_j} \equiv 0 \bmod p^T \\ 
p\mid s_j \Rightarrow p\mid k_j
\end{array} \right\},
$$
for any 
$(\s, \t)$ such that  $(\s, \t) = (\s^{[p]}, \t^{[p]})$.
It is clear that $N(p^T)=\#\mathcal{X}_{p, T}(\s, \t) $ when $p\nmid s_0\dots s_n$
 and that 
$
N(p^t)/ \prod_{j=0}^n s_j^{[p]}  =
\#\mathcal{X}_{p, T}(\s, \t)  
$
 when $p\nmid s_0\dots s_n$. It follows that 
 $$
 \frac{\mathfrak{S}_{ \mathbf{c}'; \s^{\m}  \w }}{\prod_{j=0}^n s_j }=
 \prod_p  
 \lim_{T \rightarrow \infty}  \frac{ \#\mathcal{X}_{p, T}(\s, \t)}{p^{nT}}=
 \prod_p \mathcal{X}_{p}(\s, \t),
 $$
 say.

Using the fact that $\t \mid  \v$  if and only if $\t^{[p]} \mid  \v^{[p]}$ for all $p$, it follows from part (i) of Lemma \ref{lemma on omega} that 
\begin{align*}
\mathfrak{S}_{ \boldsymbol{\epsilon} \mathbf{c};  \w  }
&+ \sum_{\substack { (\mathbf{s}, \mathbf{t}) \not = (\mathbf{1}, \mathbf{1})  \\  \mathbf{t} | \v }}
\varpi(\mathbf{s},\mathbf{t})
\frac{
\mathfrak{S}_{ \boldsymbol{\epsilon} \mathbf{c}; \s^{\m} \w  }}{ \prod_{j=0}^n  s_j}\\
&=
\prod_{p} \Bigg(  \# \mathcal{X}_{p}(\mathbf{1}, \mathbf{1}) +
\sum_{ \substack{   (\mathbf{s}, \mathbf{t}) \not = (\mathbf{1}, \mathbf{1}) \\  (\mathbf{s}, \mathbf{t}) = (\mathbf{s}^{[p]}, \mathbf{t}^{[p]})  \\ \t \mid \v^{[p]}  }  }
\varpi (\mathbf{s}, \mathbf{t} ) \cdot   \# \mathcal{X}_{p}(\mathbf{s}, \mathbf{t}  )
\Bigg).
\end{align*}
On the other hand, on 
appealing to  the inclusion-exclusion principle and the definition of $\varpi$, for any prime $p$
we return to \eqref{eq:alt} and see that 
\begin{align*}
\mathcal{M}_{ \boldsymbol{\epsilon}, T}(\v, p)
&=\#
\mathcal{X}_{p,T}(\mathbf{1}, \mathbf{1}) - \# \bigcup_{ \substack{ (\mathbf{g}; p; I) \\  \b(\mathbf{g}; p; I) \mid  \v^{[p]}    } }
\mathcal{X}_{p,T}(\mathbf{a}( \mathbf{g}; {p}; I), \mathbf{b} (\mathbf{g}; {p};{I}) )
\\
&= \# \mathcal{X}_{p,T}(\mathbf{1}, \mathbf{1}) + \sum_{k=1}^{\infty} (-1)^{k} \sum_{ \substack{ \# \mathcal{R} = k \\  \mathcal{R} = \mathcal{R}(p)  \\ \b(\mathcal{R}) \mid  \v^{[p]}  }  } \# \mathcal{X}_{p,T}(\mathbf{a}(\mathcal{R}), \mathbf{b}(\mathcal{R}) )
\\
&= \# \mathcal{X}_{p,T}(\mathbf{1}, \mathbf{1}) +
\sum_{ \substack{   (\mathbf{s}, \mathbf{t}) \not = (\mathbf{1}, \mathbf{1}) \\  (\mathbf{s}, \mathbf{t}) = (\mathbf{s}^{[p]}, \mathbf{t}^{[p]})  \\ \t \mid \v^{[p]}  }  }
\varpi (\mathbf{s}, \mathbf{t} ) \cdot   \# \mathcal{X}_{p,T}(\mathbf{s}, \mathbf{t}  ).
\end{align*}
Dividing by $p^{nT}$ and taking the limit $T\to \infty$, we are now easily led to the proof of the  claim \eqref{eq:airport}.

\section{Thin sets: proof of Theorem \ref{t:2}}\label{s:thin}

Let   $\Gamma = \sum_{j=0}^n \frac{1}{m_j} - 1$, as in \eqref{defnGamma}.
In this section we assume that
\eqref{cond1} holds and we let  $\Omega\subset \PP^{n-1}(\QQ)$ be a thin set.  Theorem
\ref{t:2} is  concerned
with an upper bound for the quantity
$$
N_{\Omega}(\PP^{n-1},\Delta;B) =
\frac{1}{2}\#\left\{\x\in \ZZ_{\neq 0}^{n+1}:
\begin{array}{l}
\gcd(x_0,\dots,x_{n-1})=1\\
|\x|\leq B,
\text{ $x_i$ is $m_i$-full $\forall ~0\leq i\leq n$}\\
c_0x_0+\dots+c_{n-1}x_{n-1}=x_n\\
(x_0:\dots:x_{n-1})\in \Omega
\end{array}{}
\right\},
$$
under the conditions on $\Omega$ that are stated in the theorem.
Let us write $N_{\Omega}(B)=
N_{\Omega}(\PP^{n-1},\Delta;B)$ to ease notation.
All of the implied constants in this section are allowed to depend on the thin set $\Omega$.

We shall proceed by using information about the size of thin sets modulo $p$
on a set of primes $p$ of positive density. Our thin set $\Omega$ is contained in a finite union
$\bigcup_{i=1}^t \Omega_i$
of thin
subsets of type I and type II.
We shall abuse notation and write $\Omega_i(\FF_p)$ for the image of $\Omega_i$ in $\PP^{n-1}(\FF_p)$ under reduction modulo $p$. Similarly, we shall write
$\widehat{\Omega}_i(\mathbb{F}_p)$ for the set of $\FF_p$-points on the affine cone over this set of points.

Let $\Omega_i \subset \mathbb{P}^{n-1}(\mathbb{Q})$ be a thin subset of type I. Then it follows from the Lang-Weil estimates \cite{LW} that there exits $C_1 > 0$ such that
\begin{eqnarray}
\label{type I est}
\# \Omega_i(\mathbb{F}_p) \leq C_1 p^{n-2},
\end{eqnarray}
for every sufficiently large prime $p$. If $\Omega_i \subset \mathbb{P}^{n-1}(\mathbb{Q})$ is a thin subset of type II, then
according to  Serre
\cite[Thm.~3.6.2]{Ser08}
 there exists a constant  $\kappa \in (0,1)$ such that
\begin{eqnarray}
\label{type II est}
\# \Omega_i(\mathbb{F}_p) \leq \kappa p^{n-1},
\end{eqnarray}
for every sufficiently large  prime $p\in \mathsf{P}_{\Omega_i}$, in the notation introduced before the statement of
Theorem \ref{t:2}.

We take advantage of this information by noticing  that
$$
N_{\Omega}(B)\leq
\sum_{i=1}^t
\#\left\{\x\in \ZZ_{\neq 0}^{n+1}:
\begin{array}{l}
\gcd(x_0,\dots,x_{n-1})=1\\
|\x|\leq B,
\text{ $x_i$ is $m_i$-full $\forall ~0\leq i\leq n$}\\
c_0x_0+\dots+c_{n-1}x_{n-1}=x_n\\
(x_0:\dots:x_{n-1}) \bmod p \in \Omega_i(\FF_p) \text{ $\forall p\in \mathcal{S}_i$}
\end{array}{}
\right\},
$$
for any finite subset of primes $\mathcal{S}_i$.
We stipulate that
$\min_{p \in \mathcal{S}_i} p$ is greater than some absolute constant depending only on $\prod_{i=0}^n |c_i| m_i$ and the thin subset $\Omega_i$.
Let $$H_i = \prod_{p \in \mathcal{S}_i} p$$
and put
$$
\Omega_{H_i}= \prod_{p \in \mathcal{S}_i} \widehat{\Omega}_i(\mathbb{F}_p).
$$
Given $\mathbf{b}' = (b_0, \ldots, b_{n-1})$ we let
\begin{equation}\label{eq:bn}
b_n = c_0 b_0 + \dots + c_{n-1} b_{n-1}
\end{equation}
and we put $\mathbf{b} = (b_0, \ldots, b_n)$.
Appealing to \eqref{eq:sign} and putting $c_n=-1$, we deduce that
$N_{\Omega}(B)
$ is
\begin{align*}
&\leq
\sum_{i=1}^t
\sum_{\mathbf{b}' \in \Omega_{H_i} }
\#\left\{\x\in \ZZ_{\neq 0}^{n+1}:
\begin{array}{l}
\gcd(x_0,\dots,x_{n})=1, ~|\x|\leq B \\
x_{j} = \pm u_{j}^{m_j} \prod_{r=1}^{m_j-1} v_{j,r}^{m_j+r} ~\forall ~0\leq j\leq n\\
  \mu^2(v_{j,r}) = 1,~ \gcd(v_{j,r}, v_{j,r'}) = 1\\
c_0x_0+\dots+c_{n}x_{n}=0\\
\x\equiv \b \bmod{H_i}
\end{array}{}
\right\}\\
&\leq
\sum_{i=1}^t
\sum_{\boldsymbol{\epsilon}\in \{\pm 1\}^{n+1}}
\hspace{-0.2cm}
\sideset{}{^{(1)}}\sum_{\mathbf{v}}
\hspace{-0.2cm}
\sum_{\mathbf{b}' \in \Omega_{H_i} }
\hspace{-0.1cm}
\#\left\{\mathbf{u}\in \NN^{n+1}:
\begin{array}{l}
\gcd(u_0w_0,\dots,u_{n}w_n)=1\\
u_j^{m_j}w_j\leq B, \text{ $0\leq j\leq n$}\\
\sum_{0\leq j\leq n}
\epsilon_j c_j w_ju_j^{m_j}=0\\
u_j^{m_j}w_j\equiv b_j \bmod{H_i}, \text{ $0\leq j\leq n$}
\end{array}{}
\right\},
\end{align*}
where
$w_j=\prod_{r=1}^{m_j-1} {v}_{j,r}^{m_j+r}$ and
$
\sum^{(1)}_{\mathbf{v}}
$
denotes a sum over $\mathbf{v}=(v_0,\dots,v_n)\in \NN^{n+1}$ satisfying $w_{j} \leq B$ and the coprimality conditions
\begin{equation}\label{eq:sum1}
\gcd \left( {v}_{j,r}, {v}_{j,r'} \right) = 1, \quad  \mu^2 \left( {v}_{j,r} \right) = 1,\quad
\gcd(w_0, \ldots, w_n) = 1.
\end{equation}
(This should not be confused with the notation $\sum^{(1)}_{\mathbf{v}}$ in
\S \ref{s:manin}, in which the condition $\gcd(w_0, \ldots, w_n) = 1
$ does not appear.)

Let us define
$$
\Omega_{\w; H_i}^{(i)} = \prod_{p \in \mathcal{S}_i} \Omega_{\w; p}^{(i)},
$$
where
$$
\Omega_{\w; p}^{(i)} = \left\{ \h \in \FF_p^{n+1} \backslash \{ \mathbf{0} \} :
\begin{array}{l}
h^{m_j}_j w_j \equiv  b_j \bmod{p} \text{ for $0 \leq j \leq n$}\\
\text{for some
$\b' \in \widehat{\Omega}_i(\mathbb{F}_p)$
}
\end{array}
  \right\}.
$$
In view of the coprimality conditions we are only interested in
$\b \not \equiv \mathbf{0} \bmod{p}$ for all $p \in \mathcal{S}_i$.
Thus, for each $p \in \mathcal{S}_i$
and $\mathbf{h}\in \Omega_{\w; p}^{(i)}$
 we have
\begin{eqnarray}
\label{cond nonsing}
h_j^{m_j}w_j \not \equiv 0 \bmod{p} \quad \text{ for some $j\in \{0,\dots,n\}$.}
\end{eqnarray}

With this notation, we may write
$$
N_{\Omega}(B)
\leq
\sum_{i=1}^t
\sum_{\boldsymbol{\epsilon}\in \{\pm 1\}^{n+1}}
\hspace{-0.2cm}
\sideset{}{^{(1)}}\sum_{\mathbf{v}}
\hspace{-0.2cm}
\sum_{\mathbf{h} \in \Omega_{\mathbf{w};H_i}^{(i)}  }
\hspace{-0.2cm}
\#\left\{\mathbf{u}\in \NN^{n+1}:
\begin{array}{l}
u_j^{m_j}w_j\leq B, \text{ $0\leq j\leq n$}\\
\sum_{0\leq j\leq n}
\epsilon_j c_j w_ju_j^{m_j}=0\\
\mathbf{u}\equiv \mathbf{h}\bmod{H_i}
\end{array}{}
\right\}.
$$
Note that $\#\Omega_{\mathbf{w};H_i}^{(i)}\leq H_i^{n+1}$.
We now seek to apply Theorem \ref{t:M} to the inner sum, much as  in \eqref{eq:S2'}.
Let $\eta>0$ be sufficiently small and assume that  $\delta$ is chosen so that
$(2n+5)\delta=\frac{1}{m_*(m_*+1)}-3\eta$, 
where we have found it convenient to set 
$m_* = \max_{0 \leq j \leq n} m_j$.  
This is plainly  satisfactory for \eqref{deltass1}.  
We take $\varepsilon=\eta$ in the statement of
Theorem \ref{t:M} and
we assume  that
$H_i$ satisfies
\begin{eqnarray}
\label{H cond}
H_i^{2(n+1)} \leq\min \{B^{  \frac{\delta}{m_*(m_*+1)}-\ve-\eta},  B^{ \frac{1}{m_*(m_*+1)}-(2n+5)\delta-\ve-\eta } \}=B^{\eta},
\end{eqnarray}
where the second equality is true provided that $\eta$ is small enough in terms of $m_*$ and $n$. Under this assumption
it can be verified that the overall contribution from the error term in Theorem \ref{t:M}  is $O(B^{\Gamma - \eta})$. It follows that
\begin{equation}\label{Nthinbound2}
N_{\Omega}(B)
\ll
\sum_{i=1}^t
\frac{B^{\Gamma}}{H_i^{n+1}}\sum_{\boldsymbol{\epsilon} \in \{ \pm 1 \}^n }
\sideset{}{^{(1)}}\sum_{\mathbf{v}}
\sum_{ \mathbf{h} \in \Omega_{\w; H_i}^{(i)}} \frac{\mathfrak{S}_{  \boldsymbol{\epsilon} \mathbf{c}  ;    { \w } } (\mathbf{h}, H_i;0) }{ \prod_{j=0}^{n}  w_j^{1/m_j} }
+B^{\Gamma - \eta},
\end{equation}
since $\mathfrak{J}_{\boldsymbol{\epsilon} \mathbf{c} }\ll 1$.

Before proceeding with an analysis of the singular series, we first  record some estimates  for the size of  $\Omega_{\w;p}^{(i)}$,
for any    $i\in \{1,\dots,t\}$.

\begin{lem}
\label{lem:prang}
We have
$\#\Omega_{\w; p}^{(i)}\leq m_* p^n$
for any $p\in \mathcal{S}_i$.
\end{lem}
\begin{proof}
Suppose without loss of generality that $p\nmid w_0$ and let $h_1,\dots,h_n\in \FF_p$ be such that
$h_j^{m_j}w_j\equiv b_j \bmod{p}$ for $1\leq j\leq n$, for some $\mathbf{b}'$, where
$b_0=c_0^{-1}(b_n-c_1b_1-\dots -c_{n-1}b_{n-1})$. Then there are at most $m_0\leq m_*$ choices for $h_0$. This confirms the  lemma.
\end{proof}

\begin{lem}
\label{lem:omega bound1''}
Assume that $p\in
 \mathcal{S}_i$ and
 $p\nmid w_j $ for $0 \leq j \leq n$.
Then
we have
$$
\# \Omega_{\w;p}^{(i)}
\leq \begin{cases}
(p-1) \#\Omega_i( \FF_p ) &\text{ if
 $p\in \mathsf{Q}_{\mathbf{m}}$,}\\
 m_*^n (p-1) \#\Omega_i( \FF_p ) &\text{ otherwise},
\end{cases}
$$
where
$\mathsf{Q}_{\mathbf{m}}$ is defined in  \eqref{eq:Q}.
\end{lem}

\begin{proof}
Let $ z \in \Omega_i(\FF_p)$.
Either there are no points in  $\Omega_{\w;p}^{(i)}$ corresponding to $z$, or else
we may assume that there exists $\mathbf{h}\in \FF_p^{n+1}\setminus \{\0\}$ such that
$$
b_j \equiv h_j^{m_j} w_j \bmod{p}
$$
for $0 \leq j \leq n$, for some
 $(b_0, \ldots, b_{n}) \in \FF_p^{n+1}$ such that $(b_0: \cdots : b_{n-1}) = z$, in which $b_n$ satisfies \eqref{eq:bn}.
   Then the number of points in
 $\Omega_{\w;p}^{(i)}$ associated to $z$ is at most the number of
 $a \in \FF_p^{*}$ and $\mathbf{k}\in \FF_p^{n+1}$ such that
$a b_j \equiv k_j^{m_j} w_j \bmod{p}$ for $0 \leq j \leq n$.
For fixed $a\in \FF_p^*$,  since $w_j \not \equiv 0 \bmod{p}$ for $0 \leq j \leq n$, it follows that
the number of  $\k$ is precisely the number of solutions to the set of congruences
$$
a h_j^{m_j} \equiv k_j^{m_j} \bmod{p},
$$
for $0 \leq j \leq n$.

If $b_j=0$ then it forces $h_j= 0$, and so  $k_j=0$. Suppose  without loss of generality that  $b_j\neq 0$ for $0 \leq j \leq R$ and  $b_{R+1}=\dots=b_n=0$. Let us fix a choice of a primitive element $g \in \FF_p^{*}$
and put  $a = g^u$,  where $1 \leq u \leq p-1$.
Then
\begin{align*}
\# \{ \k \in &\FF_p^{n+1}: a h_j^{m_j} \equiv k_j^{m_j} \bmod{p} \text{ for $0 \leq j \leq n$} \}
\\
&=
\# \{  (x_0, \ldots, x_R) \in (\FF_p^*)^{R+1} : x^{m_j}= a  \text{ for $0 \leq j \leq R$} \}
\\
&=\# \left\{
(\ell_0, \ldots, \ell_R) \in (\ZZ / (p-1) \ZZ)^{R+1} :
\begin{array}{l}
m_j  \ell_j \equiv u \bmod{p-1}\\
\text{for $0 \leq j \leq R$}
\end{array}
\right\}\\
&=
\begin{cases}
         \prod_{j=0}^R  \gcd( m_j, p-1 )
         & \text{if $\gcd( m_j, p-1 ) \mid  u$  for   $0 \leq j \leq R$,}\\
         0
         &\text{otherwise.}
\end{cases}.
\end{align*}
In this way we see that
\begin{align*}
\sum_{a \in \FF_p^{\times}} \# \{ \k \in \FF_p^{n+1}: ~&a h_j^{m_j} \equiv k_j^{m_j} \bmod{p} \text{ for $0 \leq j \leq n$} \}\\
&=\frac{\prod_{0\leq i\leq n}\gcd(m_i,p-1)}{
\mathrm{lcm}
\left(\gcd(m_0,p-1),\dots,\gcd(m_n,p-1)\right)}
(p-1).
\end{align*}
The factor in front of $(p-1)$ is $1$ when  $p\in \mathsf{Q}_{\mathbf{m}}$  and at most
$m_*^{n}$ in general. The statement of the lemma now follows.
\end{proof}

We are now ready to analyse the singular series in
\eqref{Nthinbound2}.
Let us put $c'_j = \varepsilon_j c_j$  for indices  $0 \leq j \leq n$.  We recall from \eqref{eq:SS} that
$$
\mathfrak{S}_{ \mathbf{c}' ; {\w}  }(\mathbf{h},{H_i};0) = \sum_{q=1}^\infty \frac{1}{q^{n+1}}
\sum_{\substack{ 0 \leq a < q \\   \gcd(a,q) = 1 } }
\ \prod_{j=0}^n \
 \sum_{ 0 \leq k < q  }
e\left(\frac{a}{q}  c_{j}' w_{j}   (H_i k + h_{j})^{m_j}  \right).
$$
Put
$$
\mathcal{B}_{\w}(p^t) = \frac{1}{p^{t(n+1)} }
\sum_{\substack{ 0 \leq a < p^t \\   \gcd(a,p^t) = 1 } }
\prod_{j=0}^n \  \sum_{ 0 \leq k < p^t  }
e \left(\frac{a}{p^t} c'_{j} {w_{j}} (H_i k + h_{j})^{m_j}  \right),
$$
so that
\begin{equation}
\label{ss1}
\mathfrak{S}_{ \mathbf{c}';  \w }(\mathbf{h}, H_i;0)
= \prod_{p} \left( 1 + \sum_{t=1}^{\infty} \mathcal{B}_{\w}(p^t) \right).
\end{equation}
If $p \nmid H_i$ then
$$
1 + \sum_{t=1}^{\infty} \mathcal{B}_{\w}(p^t)  =1 + \sum_{t=1}^{\infty} \frac{1}{p^{t(n+1)} }
\sum_{\substack{ 0 \leq a < p^t \\   \gcd(a,p^t) = 1 } }
\prod_{j=0}^n \  \sum_{ 0 \leq k < p^t  } e \left(\frac{a}{p^t} c'_{j} w_{j} k^{m_j}  \right).
$$
It now follows from \eqref{eq:vaughan} that
\begin{align*}
\prod_{p \nmid H_i} \left( 1 + \sum_{t = 1}^{\infty} \mathcal{B}_{\w}(p^t) \right)
&=
\sum_{\substack{q = 1 \\ \gcd(q,H_i) = 1 } }^\infty \frac{1}{q^{n+1}}
\sum_{\substack{ 0 \leq a <q \\   \gcd(a,q) = 1 } }
\prod_{j=0}^n \  \sum_{ 0 \leq k < q  }
e\left(\frac{a}{q} c'_{j} w_j k^{m_j}  \right)
\\
&\ll
\sum_{\substack{q = 1 \\ \gcd(q,H_i) = 1 } }^\infty
q^{1 - \sum_{j=0}^n \frac{1}{m_j} + \varepsilon }
\prod_{j=0}^n \gcd(q,w_j)^{\frac{1}{m_j}},
\end{align*}
for any $\varepsilon>0$.
Moreover, in the usual way, for any prime $p$ we have
\begin{equation}
\label{ss1'}
1 + \sum_{t=1}^{T} \mathcal{B}_{\w}(p^t)=
p^{-n T} N(p^T),
\end{equation}
 where
$$
N(p^T)=  \# \left\{  \k \bmod{p^{T}} : \sum_{j=0}^n c'_j  {w_j} (H_i k_j + h_j)^{m_j}\equiv 0
 \bmod p^{T}  \right\}.
$$
In order to  deal with primes  $p \mid H_i$, we
require the following simple form of Hensel's lemma.

\begin{lem}
\label{lem Hens}
Let $m, y, T \in \mathbb{N}$ and let $A, B \in \mathbb{Z}$. Assume that  $p$ is prime such that  $p \nmid A m  y$ and
$A y^m + B \equiv 0 \bmod p$.  Then
$$
\# \{ x \bmod p^T :  A x^m + B \equiv 0 \bmod p^T,  ~x \equiv y \bmod p  \} = 1.
$$
\end{lem}

Let   $p\mid H_i$. Then $H_i=pH_i'$ for some $H_i'\in \NN$ that is coprime to $p$.
It readily follows that
$$
N(p^T)
= p^{n+1} \# \left\{  \k \bmod p^{T - 1} :
\sum_{j=0}^n c'_j  {w_j} (p k_j + h_j)^{m_j} \equiv 0 \bmod p^{T}  \right\}.
$$
If  $\mathbf{h} \bmod p$ is a solution to the congruence
$$
c'_0  {w_0} h_0^{m_0}+\cdots+c'_n  {w_n} h_n^{m_n} \equiv 0 \bmod p,
$$
then necessarily it is a non-singular solution by   (\ref{cond nonsing}), since each prime $p\mid H_i$ is  large enough that $p\nmid \prod_j c_j'm_j$.
Hence for $T > 1$ it follows from Lemma \ref{lem Hens} that
$
N(p^T)=
p^{n+1} p^{n (T - 1)}=p^{nT+1}.
$
Bringing this together with  (\ref{ss1}) and  (\ref{ss1'}) we conclude that
\begin{align*}
\mathfrak{S}_{ \mathbf{c}';  \w }(\mathbf{h},{H_i};0)
&=
H_i \prod_{p \nmid H_i}   \left( 1 + \sum_{t=1}^{\infty} \mathcal{B}_{\w}(p^t) \right)\\
&\ll H_i \sum_{\substack{q = 1 \\ \gcd(q,H_i) = 1 } }^\infty
q^{1 - \sum_{j=0}^n \frac{1}{m_j} + \varepsilon }
\prod_{j=0}^n \gcd(q,w_j)^{\frac{1}{m_j}}.
\end{align*}
Inserting this into
\eqref{Nthinbound2}, our work so far has shown that
\begin{equation}\label{eq:dark}
N_{\Omega}(B)
\ll
B^{\Gamma} \sum_{i=1}^t U(B,H_i)
+B^{\Gamma - \eta},
\end{equation}
for any $\eta>0$,
where
$$
U(B,H_i)=
\frac{1}{H_i^{n}}
\sideset{}{^{(1)}}\sum_{\mathbf{v}}
\hspace{-0.2cm}
\sum_{ \mathbf{h} \in \Omega_{\w; H_i}^{(i)}}
\sum_{\substack{q = 1\\\gcd(q,H_i)=1 } }^\infty
\hspace{-0.4cm}
q^{1 - \sum_{j=0}^n \frac{1}{m_j} + \varepsilon }
\prod_{j=0}^n \frac{\gcd(q,w_j)^{\frac{1}{m_j}}}{w_j^{1/m_j} }.
$$
Let $1\leq i\leq t$ and recall that $H_i=\prod_{p\in \mathcal{S}_i}p$.
Appealing to
\eqref{type I est},
(\ref{type II est}), together with
 Lemmas \ref{lem:prang} and
\ref{lem:omega bound1''}, we deduce that
$$
\# \Omega_{\w;p}^{(i)} \leq
\begin{cases}
         C_1 m_*^n p^{n-1}
         &\text{if $p\nmid w_j$ $\forall~j$ and $\Omega_i$ is type I},\\
         \kappa p^n
         &\text{if $p\nmid w_j$  $\forall~ j$,  $\Omega_i$ is type II and $p\in \mathsf{P}_{\Omega_i}\cap \mathsf{Q}_{\mathbf{m}}$},\\
         m_* p^{n}
         &\text{otherwise},
    \end{cases}
$$
for some $\kappa\in (0,1).$

Suppose first that  $\Omega_i$ is type I
and let $\omega(H_i)=\#\mathcal{S}_i$.
Then
\begin{align*}
\frac{\#\Omega_{\w; H_i}^{(i)}}{H_i^n}\leq
\prod_{\substack{p\mid H_i\\ p\nmid w_0\dots w_n}}
\frac{C_1 m_*^n}{ p }
\prod_{\substack{p\mid H_i\\ p\mid w_0\dots w_n}}
m_*
&=  \frac{(C_1m_*^n)^{\omega(H_i)}}{H_i}
\prod_{\substack{p\mid \gcd(H_i,  w_0\dots w_n)}}
\frac{pm_* }{C_1m_*^n}\\
&\ll   H_i^{-1+\ve}
\gcd(H_i,  w_0\dots w_n),
\end{align*}
for any $\ve>0$. But then it follows that
$$
U(B,H_i)\ll H_i^{-1+\ve}
\sideset{}{^{(1)}}\sum_{\mathbf{v}}
\hspace{-0.2cm}
\gcd(H_i,  w_0\dots w_n)
\hspace{-0.4cm}
\sum_{\substack{q = 1\\\gcd(q,H_i)=1 } }^\infty
\hspace{-0.4cm}
q^{1 - \sum_{j=0}^n \frac{1}{m_j} + \varepsilon }
\prod_{j=0}^n \frac{\gcd(q,w_j)^{\frac{1}{m_j}}}{w_j^{1/m_j} }.
$$
Observe that  $\val_p(w_j) \geq  m_j + 1$ whenever
 $p \mid  w_j$. In this way we can see that
 $$
\val_p\left(\frac{ \gcd(H_i,  w_0\dots w_n)}{
\prod_{j=0}^n w_j^{1/m_j}}\right)\leq \frac{-1}{m_*}
 $$
 for every prime $p$ such that $p\mid H_i$ and $p \mid  w_0 \dots w_n$. Thus,
 on removing common factors of $w_j$ with $H_i$, one easily concludes that
\begin{align*}
U(B,H_i)
&\leq H_i^{-1+\ve} g(H_i)
\sum_{\substack{
w_j\leq B\\
\gcd(w_j, H_i) = 1
\\
\text{\eqref{eq:sum1} holds}}}
\sum_{\substack{q = 1 \\ \gcd(q, H_i ) = 1} }^\infty
q^{1 - \sum_{j=0}^n \frac{1}{m_j} + \varepsilon }
\prod_{j=0}^n \frac{\gcd(q, w_j)^{\frac{1}{m_j}}}{w_j^{1/m_j} },
\end{align*}
where 
$
g(H_i)
=\prod_{p \mid  H_i} ( 1 + O(p^{-1/m_*}) ).
$ 
It is an elementary exercise to  show that 
$$
g(H_i)\leq \exp\left(\frac{C_2(\log H_i)^{1-1/m_*}}{\log\log H_i}\right),
$$
for an absolute constant $C_2>0$.
Taking $g(H_i)\ll H_i^{\ve}$
for any $\ve>0$, it now follows from 
   \eqref{eq:B} that
$
U(B,H_i)\ll_\ve H_i^{-1+ 2\ve}.
$
Once inserted into \eqref{eq:dark}  and choosing
$\mathcal{S}_i$ in such a way that
$H_i$ is a small enough power of $B$ for \eqref{H cond},
this shows that thin subsets of type I make a satisfactory overall contribution.

Suppose next  that  $\Omega_i$ is type II.
We may assume that $p\geq m_*/\kappa$ for each $p\mid H_i$.
Then
\begin{align*}
\frac{\#\Omega_{\w; H_i}^{(i)}}{H_i^n}\leq
\prod_{\substack{p\mid H_i\\ p\nmid w_0\dots w_n \\  p\in \mathsf{P}_{\Omega_i}\cap \mathsf{Q}_{\mathbf{m}}}}
\kappa
\prod_{\substack{p\mid H_i\\ p\mid w_0\dots w_n}}
m_*
&\leq
\prod_{\substack{p\mid H_i\\ p\in \mathsf{P}_{\Omega_i}\cap \mathsf{Q}_{\mathbf{m}}}}
\kappa
\prod_{\substack{p\mid \gcd(H_i,  w_0\dots w_n)}}
\frac{m_* }{\kappa}\\
&\leq
\gcd(H_i,  w_0\dots w_n)
\prod_{\substack{p\mid H_i\\ p\in \mathsf{P}_{\Omega_i}\cap \mathsf{Q}_{\mathbf{m}}}}
\kappa.
\end{align*}
We choose $\mathcal{S}_i$ to be set of primes $1\ll p\leq \log B/\log \log B$ drawn from the
set $\mathsf{P}_{\Omega_i}\cap \mathsf{Q}_{\mathbf{m}}$. In particular $H_i$ satisfies (\ref{H cond}).
Moreover, it  follows from our assumption
\eqref{eq:density} that this set of primes has positive lower density $\rho$, say. But then
$$
\prod_{\substack{p\mid H_i\\ p\in \mathsf{P}_{\Omega_i}\cap \mathsf{Q}_{\mathbf{m}}}}
\kappa 
=
\prod_{\substack{p\mid H_i}}
\kappa
\leq \left(\frac{1}{\kappa 
}\right)^{-\frac{\rho \log B}{(\log\log B)^2}}.
$$
Feeding this into the argument that we have just given  yields
$$
U(B,H_i)\ll 
\exp\left(\frac{C_2(\log H_i)^{1-1/m_*}}{\log\log H_i} -\frac{\log(1/\kappa)\rho \log B}{(\log\log B)^2}\right)
\ll \frac{1}{(\log B)^{100}},
$$
from which it follows that
the thin subsets of type II make a satisfactory overall contribution to \eqref{eq:dark}
 under the assumption \eqref{eq:density}.
This completes the proof of Theorem \ref{t:2}.

\end{document}